\documentclass{amsart}

\usepackage{amsmath,amscd,amsfonts,amsthm,amssymb,amscd,mathrsfs}
\usepackage{appendix}
\usepackage{bm}
\usepackage{cite}
\usepackage{geometry}
\usepackage{graphicx}
\usepackage[colorlinks,linkcolor=blue,citecolor=red]{hyperref}
\hypersetup{CJKbookmarks}
\usepackage{setspace}
\usepackage{url}

\allowdisplaybreaks[4]

\newtheorem{definition}{Definition}[section]

\newtheorem{theorem}{Theorem}[section]
\newtheorem{lemma}[theorem]{Lemma}
\newtheorem{proposition}{Proposition}[section]
\newtheorem{corollary}{Corollary}[theorem]

\theoremstyle{definition}
\newtheorem{remark}{Remark}[section]

\numberwithin{equation}{section}

\newcommand{\Op}{\mathbf{Op}}

\newcommand{\T}{\mathsf{T}}

\newcommand{\PM}{\mathrm{PM}}
\newcommand{\PL}{\mathrm{PL}}
\newcommand{\CM}{\mathrm{CM}}

\newcommand{\R}{\mathcal{R}}
\newcommand{\Id}{\mathrm{Id}}

\def\eps{\varepsilon}
\def\be{\begin{equation}}
\def\ee{\end{equation}}
\def\le{\leqslant}
\def\xR{\mathbb{R}}
\def\xN{\mathbb{N}}

\DeclareMathOperator{\supp}{supp}
\DeclareMathOperator{\Avg}{Avg}

\DeclareMathOperator{\diff}{d}
\def\dx{\diff \! x}
\def\be{\begin{equation}}
\def\ee{\end{equation}}

\def\defn{\mathrel{:=}}
\def\la{\left\vert}
\def\lA{\left\Vert}
\def\bla{\big\vert}

\def\le{\leq}

\def\ra{\right\vert}
\def\rA{\right\Vert}
\def\bra{\big\vert}

\def\xT{\mathbb{T}}
\def\xZ{\mathbb{Z}}

\title{KAM via Standard Fixed Point Theorems}

\author{Thomas Alazard}
\address{CNRS - Centre de Mathématiques Laurent-Schwartz, École Polytechnique, Institut Polytechnique de Paris}
\email{thomas.alazard@polytechnique.edu}

\author{Chengyang Shao}
\address{Department of Mathematics, University of Chicago}
\email{shaoc@uchicago.edu}

\begin{document}
\maketitle
\begin{abstract}
With a mere usage of well-established properties of para-differential operators, the conjugacy equations in several model KAM problems are converted to \emph{para-homological equations} solvable by standard fixed point argument. Such discovery greatly simplifies KAM proofs, renders the traditional KAM iteration steps unnecessary, and may suggest a systematic scheme of finding quasi-periodic solutions of realistic magnitude.
\end{abstract}
\begin{spacing}{1.2}

\section{Introduction}
\subsection{Organization of the Paper}
\emph{Can regularity loss caused by small denominators be overcome by a standard functional analysis argument, instead of a Newtonian/Nash-Moser scheme?}

This has been a long-standing question in the field of dynamical systems, widely assumed to have a negative outcome. Surprisingly, we provide an affirmative answer in this paper: for several model KAM problems, it is possible to reduce the conjugacy equation to standard fixed point form via para-linearization techniques developed by Bony \cite{Bony1981}. As a result, the ``KAM steps" commonly considered essential to the problems are, in fact, unnecessary.

Several consequences follow immediately. Firstly, our new approach entirely bypasses the need for accelerated convergence techniques, a common feature in previous KAM arguments, whether Newtonian or not, used to compensate for regularity loss. It leads to significant simplification of the proof for KAM-type results. Moreover, following standard fixed-point theorems, this approach should allow substantially larger perturbations compared to earlier methods, as the cost of accelerated convergence scheme is unacceptably small magnitude of perturbation. We anticipate that the para-differential approach to KAM theory may shed light on systematic scheme of finding quasi-periodic motions of \emph{realistic, physical} magnitude. Furthermore, a fixed point form of the equation may also yield insights about \emph{non-existence}.

After a brief review of the history of this topic, we present our para-differential approach to KAM theory according to the following arrangement.

Section \ref{Circular} is an essentially self-contained explanation of the core idea of our approach. We use the problem of circular maps studied by Arnold \cite{Arnold1965} as an illustrative model. A Nash-Moser (modified Newton) type proof is briefly outlined. We then introduce the concept of \emph{para-homological equation}, serving as the counterpart to the homological equation in conventional presentation of KAM theory. Formally being the para-differential version of the latter, the regularity gain due to para-differential calculus balances with the loss. This enables one to convert it to a fixed point equation, resembling the \emph{para-inverse equation} introduced by H\"{o}rmander \cite{Hormander1990}. We then provide a heuristic showing that such simplicifaction is not specific to one-dimensional problems, but rather pertains to very broad class of conjugacy problems.

In Section \ref{Quant}, we present all the para-differential calculus tools necessary for our approach: boundedness of para-product operator, composition estimate of para-products, and a refined version of J.-M. Bony's para-linearization theorem. The operator norms and magnitude of remainders are all estimated in a quantitative manner.

Section \ref{InvTor} collects the main results of this paper. We use the quantitative para-differential calculus estimates developed in Section \ref{Quant} to give new proofs of Hamiltonian conjugacy theorems. Instead of implementing any Newtonian algorithm, we will directly write down the para-homological equation and solve it by standard fixed point techniques. This is in contrast to most of the existing literature. 

In order to state these results, we first set up the basic geometric notations. We consider $\xR^n$ as the set of \emph{column} vectors. For simplicity, we identify the space of mappings from $\xT^n$ to $\xR^n$ with the space of tangent vector fields on $\xT^n$. The variable $\theta\in\xT^n$ will be used to denote the variable on the ``abstract torus". The phase space will be $\xT^n\times\xR^n$, the variable of which is denoted as $(x,y)\in \xT^n\times\xR^n$. Vector fields on the phase space will also be considered as \emph{column} vectors. We use $D$ to denote differentiation in $(x,y)$, $\nabla$ to denote gradient with respect to $(x,y)$ (giving rise to column vectors), and $\partial$ to denote differentiation in $\theta$ or $x$. If $u$ is a mapping from $\xT^n$ to the phase space, then the differential $\partial u$ is understood as a matrix of $n$ columns and $2n$ rows:
$$
\partial u=\begin{pmatrix}
\partial u^x \\
\partial u^y
\end{pmatrix}\in\bm{M}_{2n\times n}.
$$

The phase space $\xT^n\times\xR^n$ is endowed with standard symplectic form $\diff\!x\wedge dy$. The corresponding symplectic structure has matrix representation
$$
J=\begin{pmatrix}
0 & I_n \\
-I_n & 0
\end{pmatrix}\in\bm{M}_{2n\times 2n}.
$$
Given a Hamiltonian function $h(x,y)$ on the phase space $\xT^n\times\xR^n$, we denote the Hamiltonian vector field corresponding to $h$ by 
$$
X_h:=J\nabla h
=\begin{pmatrix}
\nabla_y h \\
-\nabla_x h
\end{pmatrix}.
$$

Define the ``flat" embedding of the torus $\xT^n$ to the phase space by 
\begin{equation}\label{zeta_0}
\zeta_0:\xT^n\ni\theta\mapsto\begin{pmatrix}
\theta \\ 0
\end{pmatrix}\in\xT^n\times\xR^n.
\end{equation}
The Hamiltonian function under consideration takes the form
\begin{equation}\label{h(x,y)}
h(x,y):=a_0(x)+\langle a_1(x), y\rangle+\frac{1}{2}\langle Q(x)y,y\rangle+O(|y|^3),
\end{equation}
where $a_0$ is a scalar function, $a_1$ is a $\xR^n$-valued function, and $Q$ is a symmetric matrix valued function. In other words, we consider Hamiltonian functions defined near the ``flat" embedded torus.

We also fix a Diophantine frequency vector $\omega\in\xR^n$: there is a $\sigma>0$ and $\gamma>0$ such that
\begin{equation}\label{Dio}
|k\cdot\omega|\geq\frac{1}{\gamma|k|^\sigma},
\quad
\forall k\in\mathbb{Z}^n\setminus\{0\}.
\end{equation}
Standard measure theoretic arguments ensure the abundance of Diophantine vectors in the sense of measure: if $\sigma>n-1$, then the set of those $\omega\in\xR^n$ satisfying the Diophantine condition (\ref{Dio}) with $\gamma$ exhausting all positive real numbers is a set of first Baire category of full Lebesgue measure. We then write
$$
\nabla_\omega=\sum_{j=1}^n\omega_j\partial_j
$$
for the differentiation along the parallel vector field $\omega$ on $\xT^n$.

We now state the main theorems of the paper. The function spaces involved in these statements are classical continuously differentiable spaces $C^N$, the Zygmund spaces $C^r_*$, and the Sobolev spaces $H^s$. See Subsection \ref{Notation} for the definition of these function spaces.

\begin{theorem}[Existence of Invariant Torus]\label{Thm1} 
Fix $\sigma>n-1$. Let $s>2\sigma+2+n/2+\eps$ be a fixed index, set $r=s-n/2$, and set $N_{s+r}$ to be the least integer such that $N_{s+r}>s+r$. Let $\zeta_0$, $h$ and $\omega$ be as in (\ref{zeta_0})-(\ref{Dio}), and suppose the mean value $\Avg Q$ is an invertible matrix. Set 
$$
M_Q=\max\left(\big|(\Avg Q)^{-1}\big|,|Q|_{L^\infty}\right).
$$
Define ``error of being invariant" and ``error of being integrable" as
\begin{equation}\label{Error1}
e_0:=X_h(\zeta_0)-\nabla_\omega \zeta_0,
\quad
e_1:=Q-\Avg Q.
\end{equation}
There are constants $c_0,c_1,c_2$ depending on $|h|_{C^{N_{s+r}+3}}$ and $M_Q$ with the following property. If
\begin{equation}\label{Small1}
\big\|e_0\big\|_{H^{s+2\sigma+\eps}}
\leq \gamma^{-4}c_0,
\quad
|e_1|_{C^r_*}\leq\gamma^{-2}c_1
\end{equation}
then there is an embedding $u:\xT^n\mapsto\xT^n\times\xR^n$ of class $H^s$, such that 
$$
\|u-\zeta_0\|_{H^s}\leq c_2\gamma^{2}\|e_0\|_{H^{s+2\sigma+\eps}},
$$
and $u(\xT^n)$ is an invariant torus for the flow of $h$:
\begin{equation}\label{ConjNondege}
X_h(u)-\nabla_\omega u=0
\end{equation}
\end{theorem}

We next state a translated existence theorem that allows variation of frequency. 
\begin{theorem}[Translated Conjugacy]\label{Thm2}
Fix $\sigma>n-1$. Let $s>2\sigma+2+n/2+\eps$ be a fixed index, set $r=s-n/2$, and set $N_{s+r}$ to be the least integer such that $N_{s+r}>s+r$. Let $\zeta_0,h,\omega$ be as in Theorem \ref{Thm1}. Without any non-degeneracy assumption for $Q$, define
$$
M_Q=|Q|_{L^\infty}.
$$
Write $h_\xi(x,y)=h(x,y)+\xi\cdot y$ for $\xi\in\xR^n$. Still define ``error of being invariant" and ``error of being integrable" as
\begin{equation}\label{Error2}
e_0:=X_h(\zeta_0)-\nabla_\omega \zeta_0,
\quad
e_1:=Q-\Avg Q.
\end{equation}
There are constants $c_0,c_1$ depending on $|h|_{C^{N_{s+r}+3}}$ and $M_Q$ with the following property. If
\begin{equation}\label{Small2}
\big\|e_0\big\|_{H^{s+2\sigma+\eps}}
\leq \gamma^{-4}c_0,
\quad
|e_1|_{C^r_*}\leq\gamma^{-2}c_1,
\end{equation}
then there is an embedding $u:\xT^n\mapsto\xT^n\times\xR^n$ of class $H^s$ and a constant vector $\xi\in\xR^n$, such that 
$$
\|u-\zeta_0\|_{H^s}\leq c_2\gamma^{2}\|e_0\|_{H^{s+2\sigma+\eps}},
$$
and $u(\xT^n)$ is an invariant torus for the flow of the modified Hamiltonian $h_\xi$:
\begin{equation}\label{ConjTran}
X_{h_\xi}(u)-\nabla_\omega u=0
\end{equation}
\end{theorem}
\begin{remark}
It is not hard to generalize the results to other function spaces, for example Zygmund spaces $C^r_*$ or more general Besov spaces. Although only existence results are stated, we can actually prove uniqueness and continuous dependence in both cases under more restrictive smallness assumptions. See the end of Subsection \ref{ProofofThm}. On the other hand, although the theorems are stated for almost integrable systems, the coverage in fact is much broader under standard symplectic geometric consideration. See the discussion in Subsection \ref{Miscellaneous}.
\end{remark}

Let us explain the reason of choosing these particular KAM type theorems for validation of our fixed point approach. Compared to standard account of KAM theorems solving a symplectic diffeomorphism that brings the Hamiltonian function to a normal form, starting from an approximate solution is computationally more convenient in practical applications to realistic systems. See for example \cite{CC1997,CC2007}, where the authors coducted \emph{computer assisted proof} of KAM type results for restricted three body problems of Sun-Jupiter-asteroid with \emph{realistic} physical parameters. It can be expected that our fixed point approach will yield improved results regarding these realistic physical systems, and may greatly advance the estimate of threshold of validity.

In the early stages of KAM theory, Hamiltonian conjugacy theorems were commonly formulated as structural stability theorems of KAM normal forms. This necessitated suitable action-angle variables for the system to reduce it to a perturbative form, usually a challenging task. But if one directly searches for an invariant torus near a given ``approximately invariant torus" in the phase space, such technical issue can be largely bypassed. This idea dates back to \cite{Salamon1986} and was used in, for example, \cite{CC1988,SZ1989,DGJV2005}. 

While the coverage of Theorem \ref{Thm1}-\ref{Thm2} may appear narrower than the results in \cite{DGJV2005}, it remains sufficiently broad to apply to a diverse range. Notably, Theorem \ref{Thm1} is the famous Kolmogorov invariant tori theorem (see \cite{Kolmogorov1954,Arnold1963Proof}) with strongest non-degeneracy condition. In fact, as shown in \cite{BB2015}, once the embedding $u$ is given, a symplectic coordinate close to the original one can be readily constructed near the isotropic torus $u(\xT^n)$, under which $h$ assumes a normal form and $u(\xT^n)$ is flattened to $\xT^n\times\{0\}$. Similarly, Theorem \ref{Thm2} is in fact a special case of the ``théorème de conjugaison tordue" (see Féjoz \cite{Fejoz2016}), which further implies the iso-energetic KAM theorem as a corollary. The more general ``théorème de conjugaison tordue" by Herman was proved in \cite{Fejoz2004}, which is a properly degenerate KAM theorem. It appears that the method of our paper still applies with only minor alteration. For discussion of coverage of Theorem \ref{Thm1}-\ref{Thm2} besides these obvious specific cases, see Subsection \ref{Miscellaneous}.

We briefly discuss the possibility of further generalizations. Theorem \ref{Thm2}, with suitable adaptation and extension, may find compatibility with ``KAM for PDEs", including the Craig-Wayne-Bourgain approach to Melnikov persistency theorem. See for example \cite{Kuksin1987,CW1993,Bourgain1995,Bourgain1997,Bourgain1998} (and \cite{BB2015} for the applicability of this formalism to infinite dimensions). Such an alignment could potentially introduce a novel approach of finding ``large magnitude" KAM solutions to Hamiltonian PDEs. To the author's knowledge, the first para-differential construction of \emph{periodic} solutions of PDEs involving small denominators is carried out by Delort \cite{Delort2012}. Generalization to quasi-periodic solutions is possible. Due to the technicality involved in these extensions, we confine ourselves proving Theorem \ref{Thm1}-\ref{Thm2} in this paper to keep the narration as simple as possible. Extensions to ``KAM for PDEs" are discussed with full detail in the forthcoming work \cite{AS2024}. 

\subsection{A Brief Review of History}\label{Sec1.3}
Given the extensive volume of literature associated to KAM theory and Nash-Moser techniques, it does not seem practical to provide a panoramic review within such limited space. We refer the reader to \cite{AG1991,Bost1986,CP2018,Fejoz2004,Hamilton19821,Llave2001,Poschel2009,Wayne2008} for comprehensive description of these topics. Here, we confine ourselves to a brief overview of those works directly related to the central issue of this paper; that is, whether KAM iteration can be substituded by standard fixed point argument. 

KAM theory is generally recognized to originate from Kolmogorov's 1954 report \cite{Kolmogorov1954}, although the problems it addressed had already interested mathematicians and physicists as early as the time of Poincaré: \emph{does the Poincaré-Lindstedt series converge in presence of small denominators?} Or equivalently, \emph{do quasi-periodic solutions persist upon perturbation?} Kolmogorov suggested a positive answer to this question by introducing a sequence of canonical transformations constructed via a modified Newtonian scheme. Arnold gave a detailed yet technically different proof of Kolmogorov's theorem in \cite{Arnold1963Proof}, and extended its coverage in \cite{Arnold1963Small}. Both Kolmogorov and Arnold worked in space of analytic functions. In the meantime, Moser \cite{Moser1962} was able to replace analyticity by finite differentiability in a parallel context. Hence comes the abbreviation \emph{KAM theory}.

Realizing similarity between the Newtonian scheme used by Kolmogorov-Arnold and that used by Nash \cite{Nash1956} to address the isometric embedding problem, Moser \cite{Moser19661,Moser19662} then extracted this set of methods and applied it to broader classes of nonlinear problems. Related works that set stage for a general statement in graded spaces include Sergeraert \cite{Sergeraert1970} and H\"{o}rmander \cite{Hormander1976,Hormander1977}. In the original context of dynamical conjugacy problems, Zehnder \cite{Zehnder1974,Zehnder1975,Zehnder1976} fitted them into adapted ``hard implicit function theorems". These works are commonly recognized as the origin of the \emph{Nash-Moser implicit function theorem}, which soon found its strength in nonlinear analysis. See \cite{Hamilton19821} for a systematic narration.

Nevertheless, from a practical point of view, if the goal is to systematically find periodic or quasi-periodic motions of realistic magnitude, then KAM/Nash-Moser type methods, based on Newtonian iteration, are far from satisfactory. In fact, even when solving for the zero of a single-variable function, the Newtonian method is well-known to be much more sensitive to initial estimate than ordinary fixed point schemes. The loss of regularity caused by small denominators for conjugacy problems significantly worsens the scenario. As Hénon observed \cite{Henon1966}, the original proof by Arnold yields an infamous $10^{-300}$ upper bound for strength of perturbation, which is physically not acceptable (see also \cite{Laskar2014}):

\begin{quote}{Ainsi, ces théorèmes, bien que d’un très grand intérêt théorique, ne semblent pas pouvoir en leur état actuel être appliqués à des problèmes pratiques, où les
perturbations sont toujours beaucoup plus grandes ...}
\end{quote}
\begin{quote}{Thus, these theorems, although of great theoretical interest, do not seem to be able to be applied to practical problems in their current state, where the perturbations are always much larger ...}
\end{quote}


However, Hénon also pointed out that numerical evidence suggests the persistence of quasi-periodic solutions even in the presence of strong perturbations. Numerical evidence, though it never serves as rigorous proof, could indicate the potential extension of KAM proofs to encompass stronger perturbations. There have been some successful research efforts aimed at quantifying the validity of KAM proofs. Notable among them are \cite{CC1988,CC1997,CC2007}, where the authors proved KAM stability for certain Hamiltonian systems, including restricted three body problems, with realistic physical parameter. The proof is \emph{quantitative} and \emph{computer assisted} to validate the Newtonian algorithm.

Apart from practical considerations, mathematicians are interested in avoiding Nash-Moser type schemes because they typically offer less insight into the nonlinear structure. For example, the original proof of local existence of the Ricci flow \cite{Hamilton19822} or the mean curvature flow \cite{GageHamilton1986} employed the Nash-Moser method, given the highly degenerate parabolic operators arising from linearization. However, the DeTurck technique introduced in \cite{DeTurck1983} elegantly resolved this degeneracy by fixing geometric gauge. Another example is the quasi-linear perturbation of wave equations. Klainerman \cite{Klainerman1980} \cite{Klainerman1982} proved global well-posedness for quasi-linear perturbation of wave equations in spatial dimension higher than 6. His method was a Nash-Moser type iteration involving smoothing operators that both truncates Fourier modes \emph{and} long-time ranges to deal with loss of decay. But Klainerman and Ponce \cite{KP1983} soon discovered the loss of decay could be compensated by working in Banach spaces of appropriately decaying weights. Later Klainerman \cite{Klainerman1985} introduced the \emph{vector field method}, drastically simplifying the proof of global well-posedness for wave equations. More recent development is also worth mentioning: Hintz-Vasy \cite{HintzVasy2018} proved the nonlinear stability of Kerr-de Sitter spacetime by a Nash-Moser inverse function theorem, which was later found by Fang \cite{Fang2021} to be replaceable by suitable bootstrap argument.

As for the isometric embedding problem itself, 30 years after Nash's original work, Günther \cite{Gunther1989} converted the isometric embedding problem into a standard implicit function form through careful manipulation with Laplacian. Thus the Nash-Moser method is \emph{not} necessary for its original intended context. Almost simultaneously, H\"{o}rmander \cite{Hormander1990} recognized the connection between Nash-Moser methods and para-differential calculus: they both involve dyadic decomposition of nonlinearity. H\"{o}rmander introduced \emph{para-inverse operators} for inverse function problems with loss of regularity. Although the alternative proof for Nash embedding theorem in \cite{Hormander1990} implies slightly weaker result than \cite{Gunther1989}, the general method of the former not only addresses to the vague observation \emph{Nash-Moser technique can usually be replaced by elementary methods}, but also has the potential of applying to other nonlinear problems.

There have also been alternative proofs of KAM type results. Eliasson's paper \cite{Eliasson1996} revisited the Poincaré–Lindstedt series, directly proving its convergence by implementing the delicate cancellations among the coefficients. Other proofs of KAM theorem include \cite{BGK1999}, using renormalization group acting on frequency space; \cite{KDM2007}, using a multidimensional continued fractions algorithm; \cite{Russmann2010,Poschel2011}, which replaced the Newtonian iteration scheme with a ``slowly converging" one; and \cite{BF2013,BF2014} by rational approximation of the frequency vector. Being insightful from different aspects, these alternative proofs are all still based on various types of iteration with improved convergence.

In the recent decade, the idea of replacing modified Newtonian (Nash-Moser or KAM) iteration with para-differential calculus regained some attention. To mention a few, there was a proof of local well-posedness for gravity-capillary water waves using Nash-Moser theorem in \cite{MingZhang2009}. By para-linearizing the system, however, Alazard-Burq-Zuily \cite{ABZ2011} proved the local well-posedness under significantly lower regime of regularity. In the construction of periodic solutions for dispersive differential equations on the torus, the traditional approach often employs Nash-Moser type techniques. But Delort \cite{Delort2012} successfully substituted these techniques with ordinary iterations, aided by para-differential calculus. The idea also echoed in recent advances for the study of \emph{Landau damping}, the parallel of KAM theory in statistical mechanics. Originally proved to exist by a Newtonian scheme \cite{MouVi2011}, the result was then proved in \cite{BMM2016} using para-products in Gevrey spaces instead. 

Regarding the necessity of such replacement, we particularly highlight the work of Herman \cite{Herman1985}, where the conjugacy equation for circular diffeomorphism was elegantly transformed into a fixed-point problem using Schwarz derivatives. Being surprising enough, even more is true. The technique was significantly extended by Marmi, Moussa, and Yoccoz \cite{MMY2005, MMY2012} in their proof of finite codimensional structural stability for interval exchanging maps. Traditional KAM iteration faces limitations in such settings, since the obstructions of solution of the linear homological equation become more severe in function spaces of higher regularity, as demonstrated by Forni \cite{Forni1997, Forni2021}. However, reformulating the conjugacy equation as a fixed-point problem effectively bypasses this issue, a key insight that supports the arguments in \cite{MMY2005, MMY2012}.

On the other hand, the Schwarz derivative technique appears to be highly specific to one-dimensional problems. The analogous problem on stability of invariant surfaces of geodesic flows on translation surfaces, proposed by Forni in \cite{Forni1997}, remained open until the very recent work of Forni himself \cite{Forni2025}. Forni applied the method developed in this article to establish finite codimensional structural stability of geodesic flows on translation surfaces -- a setting where traditional KAM iteration fails due to regularity constraints\footnote{The authors are grateful to Giovanni Forni for bringing this to their attention.}. The crux of Forni's proof lies in transforming the conjugacy equation into a fixed-point form via para-linearization. By resolving the open problem nearly three decades after its original introduction, such achievement might convince the reader of the merit of the para-differential approach introduced in this paper. It is the aim of this paper to rekindle the idea of ``Nash-Moser/KAM replaced by para-differential" and develop a KAM theory from this new perspective.

\subsection{Notation and Convention}\label{Notation}
The notions in this subsection are standard and can be found in any textbook on harmonic analysis, for example \cite{SteMur1993}. Throughout the paper, we do not distinguish functions defined on $\xT^n$ with functions defined on $\xR^n$ that are $2\pi$-periodic with respect to each variable. We represent a distribution $u$ on $\xT^n$ as a Fourier series:
$$
u=\sum_{k\in \xZ^n} \hat{u}(k) e^{ik\cdot x},
$$
where the Fourier coefficient
$$
\hat{u}(k)=\int_{\xT^n}u(x)e^{-ik\cdot x}\dx.
$$
We denote $\Avg u=\hat u(0)$ for the mean value of $u$ if $u\in L^1$. If $u\in L^2$, then the series converges in $L^2$.

We introduce the \emph{Littlewood-Paley decomposition} as follows:

\begin{definition}\label{LP_Decomp_Def}
Fix a function $\varphi\in C^\infty_0(\xR^n)$, 
with support in an annulus $\{1/2\le \la \xi\ra\le 2\}$, so that
$$
\sum_{j=1}^\infty\varphi(2^{-j}\xi)=1-\psi(\xi),
\quad\supp\psi\subset\{|\xi|\leq1\}.
$$
The Littlewood-Paley decomposition of a distribution $u$ on $\xT^n$ is the defined as
\begin{equation}\label{LP_Decomp}
u=\Delta_{0}u+\sum_{j\ge1}\Delta_ju,
\quad\text{where}\quad
\Delta_j u=\sum_{k\in\mathbb{Z}^n}\varphi(2^{-j}k)\hat{u}(k)e^{ik\cdot x},
\quad j\geq1,
\end{equation}
while one fixes $\Delta_0u=\hat u(0)=\Avg u$. 

The partial sum operator $S_j$ is defined as
$$
S_j=\sum_{l\leq j}\Delta_l,\quad j\geq0,
$$
while for $j\leq0$ one just fixes $S_j=\Delta_0$.
\end{definition}
The summand $\Delta_ju$ in Definition \ref{LP_Decomp_Def} is called \emph{j'th building block}. It has Fourier support contained in the dyadic annulus $\{0.5\cdot 2^j\leq |\xi|\leq 2\cdot 2^j\}$. The speed of convergence of $S_ju$ to $u$ reflects the regularity of $u$ -- and it is the content of \emph{Littlewood-Paley theory} to study this connection. 

For an index $s\in \xR$, the \emph{Sobolev space} $H^s(\xT^n)$ consists of those distributions $u$ on $\xT^n$ such that
$$
\lA u\rA_{H^s}:=\left( \sum_{k\in \xZ^n} \big(1+\la k\ra^2\big)^s \bla \hat{u}(k) \bra^2\right)^{1/2} <+\infty.
$$
The space $(H^s,\lA \cdot\rA_{H^s})$ is a Hilbert space. If $s\in \xN$, then
\[
H^s(\xT^n)=\bigl\{u\in L^2(\xT^n) : 
\forall \alpha\in \xN^n,~|\alpha|\leqslant s,~\partial_x^\alpha u\in L^2(\xT^n)\bigr\},
\]
where $\partial_x^\alpha$ is the derivative in the sense of distribution of $u$. Moreover, when $s$ is not an integer, the Sobolev spaces coincide with those obtained by interpolation. If we define
$$
\lA u\rA_{s}^2 \defn \sum_{j=0}^{\infty}2^{2js}\lA \Delta_j u\rA_{L^2}^2,
$$
then it follows from Plancherel theorem that $\lA \cdot\rA_{H^s}$ and $\lA \cdot\rA_{s}$ are equivalent.

We will be using the \emph{Zygmund spaces} (also known as \emph{Lipschitz spaces} in the literature) throughout the paper. For an index $r\in\mathbb{R}$, the Zygmund space $C^r_*$ consists of those distributions $u$ on $\xT^n$ such that
\begin{equation}\label{Zygmund}
|u|_{C^r_*}:=
\sup_{j\geq0} 2^{jr}|\Delta_ju|_{L^\infty}<+\infty.
\end{equation}
Direct manipulation with series implies $H^s\subset C^{s-n/2}_*$ for any $s\in\mathbb{R}$. For non-integer $r>0$, the $C^r_*$-norm is equivalent to the H\"{o}lder norm with index $r$, i.e. the norm
$$
|u|_{C^{[r]}}
+\sum_{|\alpha|=[r]}\sup_{x,y\in\mathbb{T}^n}\frac{\big|\partial^{\alpha}u(x)-\partial^{\alpha}u(y)\big|}{|x-y|^{r-[r]}}.
$$
Here $[r]$ is the integer part of $r$. However, when $r$ is a natural number, the space $C^r_*$ is strictly larger than the classical space of Lipschitz continuous functions.

\section{Circular Map as Illustrative Model}\label{Circular}
In this section, we will use the circular map model studied by Arnold \cite{Arnold1965} to illustrate the core idea of our paper. Although a quite complete global theory under sharp regularity and number-theoretic assumptions is available (see for example \cite{Herman1979,Yoccoz1984,KS1987,KO1989}), we find it illuminating to revisit the perturbation problem for this model from the very beginning.

Throughout this section, we identify the circle $\mathbb{S}^1=\mathbb{R}/2\pi\mathbb{Z}$, and write $\tau_\alpha:x\mapsto x+\alpha$ for the rotation of angle $\alpha$, where $\alpha\in(0,2\pi)$. Addition of arguments will all be understood as modulo $2\pi$. To avoid unnecessary confusion in notation, from now on we shall use $\eta^\iota$ to denote the inversion of a nonlinear mapping $\eta$ (should it exist), and use $a^{-1}$ to denote the reciprocal of a number $a$.

\subsection{Description of the Problem}
Let $\alpha\in(0,2\pi)$ be non-commensurable with $\pi$. Consider the following classical problem:

\begin{itemize}
    \item Suppose $f:\mathbb{S}^1\mapsto\mathbb{S}^1$ is smooth and is close to 0. Is the diffeomorphism $x\mapsto x+\alpha+f(x)$ smoothly conjugate to the rotation $\tau_\alpha$?
\end{itemize}

Of course, an integrablity condition must be posed for $f$. Recall the classical theorem due to Denjoy: 
\begin{theorem}
For a $C^\infty$ orientation-preserving diffeomorphism of $\mathbb{S}^1$ to itself, if its rotation number $\alpha$ is not commensurable with $\pi$, then it is topologically conjugate to the rotation $\tau_\alpha$.
\end{theorem}

If $x+\alpha+f(x)$ has rotation number different from $\alpha$, then by Denjoy's theorem, it definitely cannot conjugate to $\tau_\alpha$. On the other hand, Denjoy's theorem has no implication on smoothness of the conjugation. Therefore the question is not answered by Denjoy's theorem. We modify it as follows:

\begin{itemize}
    \item Suppose $f:\mathbb{S}^1\mapsto\mathbb{S}^1$ is smooth and is close to 0, such that the diffeomorphism $x\mapsto x+\alpha+f(x)$ still has rotation number $\alpha$. Is it smoothly conjugate to the rotation $\tau_\alpha$? 
\end{itemize}

Let us write the hypothetical conjugation as $\eta(x)=x+u(x)$, and try to solve the superficially more general conjugacy equation
$$
\eta(x+\alpha)=\eta(x)+\alpha+f\circ\eta(x)-\lambda
$$
for $\eta$ and the auxiliary real parameter $\lambda$\footnote{It is elementary that if $x\mapsto x+\alpha+f(x)$ has rotation number $\alpha$, and the conjugacy equation has a solution $(\eta,\lambda)$, then $\lambda$ is necessarily zero. Introducing $\lambda$ makes the presentation simpler. See for example Proposition III.4.1.1 of \cite{Herman1979}.}. A simple rearrangement converts this equation to
\begin{equation}\label{ConjRot0}
\Delta_\alpha u=f\circ(\Id+u)-\lambda,
\quad\text{where}\quad
\Delta_\alpha u:=u\circ\tau_\alpha-u.
\end{equation}
We may also compose the equation with $(\Id+u)^\iota$, and reformulate equation (\ref{ConjRot0}) as
\begin{equation}\label{ConjRot1}
\big[\Delta_\alpha u\big]\circ(\Id+u)^\iota=f-\lambda.
\end{equation}
We will discuss the technical difference between (\ref{ConjRot0}) and (\ref{ConjRot1}) in Appendix \ref{AppA}.

In order to clarify the obstacles of solving the conjugacy problem, we may linearize either (\ref{ConjRot0}) or (\ref{ConjRot1}) at $(u,\lambda)=(0,0)$ along direction $(v,\mu)$. Neglecting $f$ as well, we obtain the linear \emph{homological equation}
$$
\Delta_\alpha v+\mu=h.
$$
This linearized equation is solved via Fourier transform: normalizing $\hat v(0)=0$, the unique solution is
$$
v(x)=\sum_{k\neq0}\frac{\hat h(k)e^{ikx}}{e^{ik\alpha}-1},
\quad
\mu=\Avg h,
$$

Consequently, in order that the conjugacy problem is solvable even at the linear level, an additional number-theoretic condition must be posed for the number $\alpha$. We require that $\alpha/\pi$ is Diophantine of type $(\sigma,\gamma)$: there are  $\sigma>0$ and $\gamma>0$ such that
    \begin{equation}\label{Dio0}
    \left|\frac{q\alpha}{\pi}-p\right|\geq\frac{1}{\gamma q^{\sigma}},
    \quad
    \forall p,q\in\mathbb{Z}\setminus\{0\}.
    \end{equation}

Such numbers are abundant. Indeed, Liouville's inequality asserts that if an algebraic number $\alpha$ is of degree $D$, then the inequality $|q\alpha-p|\geq c/q^{D-1}$ must hold for some $c$. A measure-theoretic argument asserts that with $\sigma>1$ fixed, the set of $(\sigma,\gamma)$ Diophantine numbers with $\gamma$ exhausting all positive real numbers form a set of full Lebesgue measure and of first Baire category.

Obviously, if $\alpha$ satisfies the Diophantine condition (\ref{Dio0}), then $\Delta_\alpha^{-1}$ is a well-defined operator mapping functions of mean zero on $\mathbb{S}^1$ to functions of mean zero, satisfying
\begin{equation}\label{ContHs0}
\|\Delta_\alpha^{-1}f\|_{H^s}\leq C\gamma\|f\|_{H^{s+\sigma}},
\quad
\text{if }\Avg f=0.
\end{equation}
Here $C$ is an absolute constant independent of $\gamma$ or $s$. This enables one to solve the linearized equation $\Delta_\alpha v+\mu=h$, at least for very regular right-hand-side. But due to this loss of information on regularity, a usual iterative scheme to solve (\ref{ConjRot1}) will necessarily terminate after finitely many steps.

\subsection{Approximate Right Inverse}
Since the operator $\Delta_\alpha^{-1}$ causes a loss of regularity that cannot be compensated by any known elliptic technique, it is natural to employ modified Newtonian iterative scheme to solve the conjugacy problem (\ref{ConjRot0}). At a first glance it appears to be friendlier than (\ref{ConjRot1}), but in fact its linearized operator only admits an \emph{approximate right inverse}. This brings more technicality for the application of Nash-Moser scheme. Let us describe the strategy of resolving it.

We set the unknwon $U=(u,\lambda)$, and define a mapping
$$
\mathscr{F}(f,U)=\Delta_\alpha u-f\circ(\Id+u)+\lambda.
$$
The linearization of $\mathscr{F}$ at $U=(u,\lambda)$ along $V=(v,\mu)$ is
$$
D_U\mathscr{F}(f,U)V=\Delta_\alpha v-f'\circ(\Id+u)v+\mu.
$$
In the literature of dynamical systems, given $U=(u,\lambda)$, the linearized equation
\begin{equation}\label{RotLin}
D_U\mathscr{F}(f,U)V=h
\end{equation}
for the unknown $V=(v,\mu)$ with a given right-hand-side $h$ is referred to as the \emph{homological equation}. 

In general, one only expects to solve the homological equation \emph{approximately}, since the operator $\Delta_\alpha -f'\circ(\Id+u)$ is hard to invert. To explain the meaning of being ``approximately solvable", we write down a simple yet crucial identity, which is a direct consequence of the definition of $\mathscr{F}$:
\begin{equation}\label{f'F'}
f'\circ(\Id+u)=\frac{\Delta_\alpha u'}{1+u'}-\frac{[\mathscr{F}(f,U)]'}{1+u'}.
\end{equation}
So in fact
\begin{equation}\label{F_U(f,U)}
\begin{aligned}
D_U\mathscr{F}(f,U)V
&=\Delta_\alpha v-\frac{\Delta_\alpha u'\cdot v}{1+u'}
+\frac{[\mathscr{F}(f,U)]'}{1+u'}v+\mu\\
&=(1+u'\circ\tau_\alpha)\Delta_\alpha\left(\frac{v}{1+u'}\right)+\frac{[\mathscr{F}(f,U)]'}{1+u'}v+\mu.
\end{aligned}
\end{equation}
Thus, defining 
$$
\Psi(U)h=\left((1+u')\Delta_\alpha^{-1}\left[\frac{h-\mu(h)}{1+u'\circ\tau_\alpha}\right],\,\mu(h)\right),
\quad
\text{where }
\mu(h)=\frac{\mathrm{Avg}\big((1+u'\circ\tau_\alpha)^{-1}h\big)}{\mathrm{Avg}\big((1+u'\circ\tau_\alpha)^{-1}\big)},
$$
there holds
\begin{equation}\label{ApproxSolu}
D_U\mathscr{F}(f,U)\Psi(U)h-h
=[\mathscr{F}(f,U)]'\Delta_\alpha^{-1}\left(\frac{h}{1+u'\circ\tau_\alpha}-\mu(h)\right).
\end{equation}
Equality (\ref{ApproxSolu}) is the property defining ``approximate solvability": \emph{the linear operator $\Psi(U)$ is an exact right inverse of $D_U\mathscr{F}(f,U)$ at a precise solution $U$ of $\mathscr{F}(f,U)=0$.}

The loss of regularity caused by $\Delta_\alpha^{-1}$ is usually identified as the primary difficulty for solving $\mathscr{F}(f,U)=0$. A modified Newtonian scheme was implemented by Kolmogorov \cite{Kolmogorov1954}, Arnold \cite{Arnold1963Proof,Arnold1963Small,Arnold1965} and Moser \cite{Moser19661,Moser19662}. Being technically different, the general idea shared by them is to define a sequence $U_k$ by inductively solving a sequence of homological equations:
\begin{equation}\label{NewtonHomo}
U_{k+1}=U_k-S_k\big[\Psi(U_k)\mathscr{F}(f,U_k)\big].
\end{equation}
Here the operator $S_k$ is either the restriction operator to a smaller domain in case all functions involved are analytic, or a smoothing operator in case all functions are of finite differentiability. The quadratic convergence property of Newtonian scheme cancels the large constants so produced and ensures the convergence of $U_k$ to a genuine solution. Proving this convergence is where the complexity accumulates.

The issue with approximate invertibility was first noticed by Zehnder \cite{Zehnder1974,Zehnder1975,Zehnder1976}, who introduced adapted hard implicit function theorems (Nash-Moser type theorems) to fit such conjugacy problems into a unified formalism. The proofs of these hard implicit function theorems still rely on modified Newtonian schemes, essentially being (\ref{NewtonHomo}). For the statement and detailed proof of Nash-Moser theorem with either exactly or approximately invertible linearized operator, see \cite{Zehnder1974,Zehnder1975,Zehnder1976,Hamilton19821}. Though taking various forms on various graded spaces, all versions of Nash-Moser type theorems share the common features of \emph{restoring regularity through smoothing operators} and \emph{ensuring convergence through quadratic property of Newtonian algorithm}. While being universal and powerful theoretically, the practical disadvantage of this approach is that the allowed magnitude of perturbation is usually extremely small. A different approach is required if the aim is to find solutions of physically reasonable size.

\subsection{Quick Introduction to Para-differential Calculus}\label{Quick}
Surprisingly, if para-differential calculus is used to attack on (\ref{ConjRot0}), we can drastically simplify the proof of existence of solution. For simplicity, we do not present a panorama of para-differential calculus in this article; instead, we list the propositions that suffice for our exposition, along with brief descriptions of the heuristics behind them. So let us first introduce the \emph{para-product operators}, the central objects of this paper.

\begin{definition}\label{ParaMult_Def}
Fix Littlewood-Paley decomposition as in Definition \ref{LP_Decomp_Def}. Given a distribution $a=a(x)$ on $\xT^n$, the para-product operator $T_a=\Op^\PM(a)$ associated to $a$ is defined as, regardless of the meaning of convergence,
\begin{equation}\label{ParaMult}
T_au=\Op^\PM(a)u:=\sum_{j\geq0} S_{j-3}a\cdot \Delta_ju.
\end{equation}
\end{definition}
For $j\geq1$, notice that the summand $S_{j-3}a\cdot \Delta_ju$ in (\ref{ParaMult}) has its Fourier support contained in the annulus $\{0.25\cdot 2^j \leq|\xi|\leq 2.25\cdot 2^j\}$. A standard construction is the \emph{para-product decomposition}:
$$
au=T_au+T_ua+R_\PM(a,u),
\quad 
R_\PM(a,u)=\sum_{j,k:|j-k|<3}\Delta_ja\cdot\Delta_ku,
$$
as long as the right-hand-side is well-defined in some suitable function space. Such decomposition separates apart \emph{high-low}, \emph{low-high} and \emph{high-high} frequency interactions in a given multiplication $au$, and is therefore widely used in the study of harmonic analysis and dispersive partial differential equations. We refer the reader to Section 4.4 of \cite{Grafakos2009} for a comprehensive description of the role it plays in harmonic analysis. \cite{KPV1993,Kato1995,Staffilani1995} are typical examples of application of these constructions in dispersive PDEs.

The following results on para-product operators are directly cited from \cite{Bony1981} (see also Chapters 8-10 of \cite{Hormander1997}). They are the three fundamental ingredients of para-differential calculus.

\begin{proposition}[Continuity of para-product, rough version]\label{ContRough}
If $a\in L^\infty(\xT^n)$, then $T_a$ is a bounded linear operator from $H^s$ to $H^s$ for all $s\in\mathbb{R}$, and in fact
$$
\|T_a\|_{\mathcal{L}(H^s,H^s)}\lesssim_s|a|_{L^\infty}.
$$
\end{proposition}
\begin{proposition}[Composition of para-products, rough version]\label{APRough}
If $a,b\in C_*^r$ with $r>0$, then $T_aT_b-T_{ab}$ is a bounded linear operator from $H^s$ to $H^{s+r}$ for all $s\in\mathbb{R}$. Furthermore, $T_aT_b-T_{ab}$ is continuously bilinear in $a,b\in C_*^r$:
$$
\big\|T_aT_b-T_{ab}\big\|_{\mathcal{L}(H^s,H^{s+r})}
\lesssim_{s,r}|a|_{C^r_*}|b|_{C^r_*}.
$$
\end{proposition}

\begin{proposition}[Para-linearization, rough version]\label{PLRough}
Suppose $r=s-n/2>0$, $u\in H^s(\mathbb{T}^n;\mathbb{R}^L)$. Let $N_{s+r}$ be the least integer such that $N_{s+r}>s+r$. Suppose $F=F(x,z)\in C^{N_{s+r}+2}(\mathbb{T}^n\times\mathbb{R}^L)$. Then there holds the following para-linearization formula:
$$
F(x,u)-F(x,0)
=T_{F'_z(x,u)}u
+\R_{\PL}\big(F(x,\cdot),u\big)u
\in H^s+H^{s+r},
$$
where $\R_{\PL}\big(F(x,\cdot),u\big)$ is a bounded linear operators $H^s\mapsto H^{s+r}$, so that
$$
\begin{aligned}
\left\|\R_{\PL}\big(F(x,\cdot),u\big)\right\|_{\mathcal{L}(H^s,H^{s+r})}
&\leq |F|_{C^{N_{s+r}+2}}\big(1+\|u\|_{H^s}\big).
\end{aligned}
$$
Moreover, $\R_{\PL}\big(F(x,\cdot),u\big)\in\mathcal{L}(H^s,H^{s+r})$ depends continuously on $u\in H^s$.
\end{proposition}

Precise statements of these results can be found in Section \ref{Quant}. Let us describe how Proposition \ref{ContRough}-\ref{PLRough} should be interpreted. The reason that Bony introduced para-differential calculus in \cite{Bony1981} is to understand microlocal regularity for solutions of nonlinear partial differential equations. For this purpose, it is necessary to try to separate out the \emph{most irregular part} out of a given nonlinear expression $F(x,u)$, which is the content of the para-linearization theorem \ref{PLRough}. The terminology \emph{para-linearzation} refers to the observation that the most irregular part, $T_{F'_z(x,u)u}u$ out of $F(x,u)$, is exactly given by the para-differential operator corresponding to the \emph{linearization} of $F$. 

Bony observed that $T_a$ exhibits advantages over the usual product operator: the operator $M_a\colon u\mapsto au$ is bounded on all Sobolev spaces $H^s$ only if  $a\in C^\infty$, while the boundedness of $T_a$ on $H^s$ relies on the \emph{minimal} regularity assumption that $a\in L^\infty$, as indicated by Proposition \ref{ContRough}. On the other hand, $T_a$ shares the same \emph{algebraic structure} as the usual multiplication $M_a$, up to smoothing remainders. While the mapping $a\mapsto M_a$ obviously is an algebra embedding from the function algebra $C^r_*$ to the operator algebra $\mathcal{L}(H^s,H^s)$ (with $s\leq r$), Proposition \ref{APRough} shows that the mapping $a\mapsto T_a$ is an algebra homomorphism from the function algebra $C^r_*$ to the operator algebra $\mathcal{L}(H^s,H^s)$ modulo smoothing operators for any $s\in\xR$. 

Therefore, Proposition \ref{ContRough}-\ref{PLRough} provides a procedure that enables one to study a nonlinear expression $F(x,u)$ \emph{as if it were linear}: by para-linearization, one separates out the most irregular part, which obeys the same algebraic laws of the usual (linear) pseudo-differential operators.

\subsection{Para-homological Equation}\label{ParaDiffIllu}

In this subsection, we illustrate how the conjugacy equation (\ref{ConjRot0}) can be solved in an extremely simple manner, provided that the para-differential tools from Subsection \ref{Quick} are admitted as standard, as in the field of harmonic analysis and dispersive partial differential equations. To the best of the authors' knowledge, the proof presented below stands as the shortest and simplest proof of KAM theorem for the circular map model available in literature.

\begin{proof}[Simple solution for (\ref{ConjRot0})]
We notice that the general Dirichlet approximate theorem implies $\sigma\geq1$. Fix $s\geq\sigma+1.5+\eps$. Suppose \emph{in a priori} $u\in H^s$, so that by Sobolev embedding, we have $u\in C^r_*$ with $r=s-0.5>1$. Suppose $f$ is sufficiently smooth, for example $f\in C^{N_{s+r}+2}$, where $N_{s+r}$ is the least integer strictly greater than $s+r$. 

By Proposition \ref{PLRough}, we have the para-linearization formula for $\mathscr{F}(f,U)$:
\begin{equation}\label{PLF(f,U)}
\begin{aligned}
\mathscr{F}(f,U)
&=\Delta_\alpha u-f-T_{f'\circ(\Id+u)}u
-\R_{\PL}\big(f(x+\cdot),u\big)u+\lambda\\
&=\Delta_\alpha u-T_{\Delta_\alpha u'/(1+u')}u-f
+T_{[\mathscr{F}(f,U)]'/(1+u')}u
-\R_{\PL}\big(f(x+\cdot),u\big)u+\lambda\\
&=T_{(1+u'\circ\tau_\alpha)}\Delta_\alpha T_{1/(1+u')}u
-f
+T_{[\mathscr{F}(f,U)]'/(1+u')}u
-\R_{\PL}\big(f(x+\cdot),u\big)u+R_1(u)+\lambda.
\end{aligned}
\end{equation}
Here we used the key identity (\ref{f'F'}) again. The last equality in (\ref{PLF(f,U)}) is valid since 
$$
[T_{1/(1+u')}u]\circ\tau_{\alpha}=T_{1/(1+u'\circ\tau_{\alpha})}(u\circ\tau_{\alpha}).
$$
Note that $\tau_\alpha$ commutes with any Fourier multiplier, which is the reason that this equality holds. The remainder 
$$
R_1(u)=\R_\CM(a,b_1)u+\R_\CM(a,b_2)u,
$$
where
$$
a_1=1+u'\circ\tau_\alpha,\quad
b_1=\frac{1}{1+u'\circ\tau_\alpha},\quad
b_2=-\frac{\Delta_\alpha u'}{(1+u')(1+u'\circ\tau_\alpha)},
$$
is produced when replacing $T_{ab}$ by $T_aT_b$ in view of Proposition \ref{APRough}. Formally, we obtain (\ref{PLF(f,U)}) simply by replacing all products with para-products in (\ref{F_U(f,U)}).

We then consider the \emph{para-homological equation} for the unknown $U=(u,\lambda)\in H^s\times\mathbb{R}$:
\begin{equation}\label{RotParaHom}
\begin{aligned}
T_{(1+u'\circ\tau_\alpha)}\Delta_\alpha T_{1/(1+u')}u
&=f+\R_{\PL}\big(f(x+\cdot),u\big)u-R_1(u)-\lambda.
\end{aligned}
\end{equation}
Equation (\ref{RotParaHom}) collects all terms but the one linear in $\mathscr{F}(f,U)$ in (\ref{PLF(f,U)}). It can be directly reduced to fixed point form, which we shall refer as the \emph{para-inverse equation} following H\"{o}rmander \cite{Hormander1990}:
\begin{equation}\label{RotParaInv}
\begin{aligned}
u
&=T_{1/(1+u')}^{-1}\Delta_\alpha^{-1}T_{(1+u'\circ\tau_\alpha)}^{-1}\left[f+\R_{\PL}\big(f(x+\cdot),u\big)u-R_1(u)-\lambda\right].
\end{aligned}
\end{equation}
The equation itself uniquely determines $\lambda$: it should balance out the mean value to make $\Delta_\alpha^{-1}$ applicable.

By Proposition \ref{APRough}, the remainder 
$$
R_1(u)\in H^{s+r-1}\subset H^{s+\sigma+\eps},
$$
is continuous in $u\in H^s$, and vanishes quadratically as $u\mapsto 0\ni H^s$. By proposition \ref{PLRough}, 
$$
\big\|\R_{\PL}\big(f(x+\cdot),u\big)u\big\|_{H^{s+r}}
\leq C_s|f|_{C^{N_{s+r}+2}}\big(1+\|u\|_{H^s}\big).
$$
Due to Proposition \ref{ContRough} and (\ref{ContHs0}), we find that the right-hand-side of (\ref{RotParaInv}) is a well-defined continuous mapping from $H^s$ to $H^{s+\eps}$. The para-differential remainder estimates further ensure that if $|u'|_{L^\infty}\leq \delta$ and $\|u\|_{H^s}\leq \rho\ll1$, then the $H^{s+\eps}$ norm of the right-hand-side of (\ref{RotParaInv}) is controlled by
$$
K(\delta,\rho)\gamma\left(|f|_{C^{N_{s+r}+2}}(1+\rho)
+\rho^2\right),
$$
where $K$ is an increasing function. Thus if $f$ is sufficiently close to 0, for example $|f|_{C^{N_{s+r}+2}}\ll\rho\gamma^{-1}$, the right-hand-side of (\ref{RotParaInv}) will be a continuous mapping from the closed ball $\bar B_\rho(0)\subset H^s$ to itself, with range actually in $H^{s+\eps}$. By the Schauder fixed point theorem, such a mapping has a fixed point $u$. We thus have a solution $U=(u,\lambda)$ of the para-inverse equation (\ref{RotParaInv}), hence the para-homological equation (\ref{RotParaHom}).

Given this solution $U$, the para-linearization formula (\ref{PLF(f,U)}) yields
$$
\mathscr{F}(f,U)
-T_{[\mathscr{F}(f,U)]'/(1+u')}u=0.
$$
If we consider the left-hand-side as a linear operator acting on $\mathscr{F}(f,U)\in C^1$, then by Proposition \ref{ContRough}, it follows that 
$$
\big\|T_{[\mathscr{F}(f,U)]'/(1+u')}u\big\|_{H^{s}}
\lesssim \big|\mathscr{F}(f,U)\big|_{C^1}\|u\|_{H^s}.
$$
So if $|f|_{C^{N_{s+r}+2}}$ 
is sufficiently small (hence the solution $u$ has small $H^s$ norm), a Neumann series argument forces $\mathscr{F}(f,U)=0$.
\end{proof}

\begin{remark}
Schauder's fixed point theorem implies the existence of 
solution in a non-constructive manner. If we require more differentiability of $f$, for example $f\in C^{N_{s+r}+3}$ instead of $C^{N_{s+r}+2}$, then the remainder $\R_{\PL}\big(f(x+\cdot),u\big)u$ will be continuously differentiable in $u$. In fact, we can directly compute (see e.g. \cite{Hormander1990}) the linearization of para-linearization remainder along an increment $v$ as
$$
\begin{aligned}
D_u\left[\R_{\PL}\big(F(\cdot),u\big)u\right]v
&=F'(u)v-T_{F''(u)v}u-T_{F'(u)}v\\
&=T_v\left(\R_{\PL}\big(F'(\cdot),u\big)u\right)+\text{smoother terms}.
\end{aligned}
$$
If $f\in C^{N_{s+r}+3}$, then $u\mapsto\R_{\PL}\big(f(x+\cdot),u\big)u$ is a $C^1$ mapping from $H^s$ to $H^{s+r}$. Given that $f$ is close to $0$ in $C^{N_{s+r}+3}$, a Banach fixed point argument is then applicable, since the right-hand-side of (\ref{RotParaInv}) will have Lipschitz constant $<1$ in $u$. The solution $u$ is thus the limit of a Banach fixed point iteration sequence. Furthermore, if we assume $f\in C^\infty$, then by differentiating the fixed point equation (\ref{ConjRot0}), we obtain a \emph{linear} iterative sequence that converges to the derivative of $u$. This easily yields the additional regularity $u\in C^\infty$.
\end{remark}

Formally, (\ref{RotParaHom}) is the para-differential counterpart of the linearized problem (\ref{RotLin}). Nevertheless, \emph{the gain of regularity through para-linearization balances with the loss}, and thus renders unnecessary \emph{any} convergence-improvement technique. This significantly differs from the prevailing practice in the existing KAM theory literature, with perhaps the only exception being Herman's Schwarz derivative technique; see the discussion in Subsection \ref{Sec1.3}. 

We note that the usual KAM/Nash-Moser iteration (\ref{NewtonHomo}) relies on smoothing operators that are essentially Fourier truncations, which are also used in the construction of para-differential operators as seen in (\ref{ParaMult_Def}). However, smoothing operators are not exploited to its maximal potential in KAM/Nash-Moser iteration since interactions of different frequencies are not used to balance the regularity loss caused by $\Delta_\alpha^{-1}$. This is exactly captured by para-product operators and para-linearization.

Furthermore, we emphasize that para-product operators, together with the remainders in para-linearization, all admit very explicit expressions, essentially involving only convolutions. Consequently, the size of the perturbation $f$ in (\ref{RotParaInv}), hence in the original conjugacy problem (\ref{ConjRot0}), does not have to be unreasonably small, as the proof of Banach or Schauder fixed point theorem indicates. 

\subsection{Refinement in Regularity}\label{RefinementRot}
If we take into account that $x+u(x)$ is a diffeomorphism of $\mathbb{S}^1$, then the loss of regularity can be managed more delicately using \emph{para-composition operator} introduced by Alinhac \cite{Alinhac1986}.

\begin{definition}\label{ParaComp}
Let $\chi:\xT^n\mapsto\xT^n$ be a Lipschitz diffeomorphism. The para-composition operator $\chi^\star$ associated to $\chi$ is defined by
$$
\chi^\star F:=\sum_{j\geq0}(S_{j+N}-S_{j-N})\big((\Delta_jF)\circ\chi\big).
$$
Here $N$ depends only on $|\partial\chi|_{L^\infty}$.
\end{definition}

In fact, the para-composition operator admits several equivalent definitions, see for example Appendix A of Chapter 2 in \cite{Taylor2000} or Section 5 of \cite{AM2009}; but we do not aim to elaborate them within our current paper. We directly cite several results concerning regularity of para-composition operators. The proof can be found in, for example, \cite{Alinhac1986}, Appendix A of Chapter 2 in \cite{Taylor2000}, Section 3 of \cite{ngu2016}, or Section 5 of \cite{Said2023}. Roughly speaking, the para-composition operator captures the ``irregularity" of the para-linearization remainder $F(u)-T_{F'(u)}u$.

\begin{proposition}\label{PCRough}
If $\chi:\xT^n\mapsto\xT^n$ is a Lipschitz diffeomorphism, then $\chi^\star$ is a bounded linear operator from $H^s$ to $H^s$ for all $s\in\mathbb{R}$:
$$
\|\chi^\star\|_{\mathcal{L}(H^s,H^s)}
\leq K_s\big(|\partial\chi|_{L^\infty},|\partial\chi^{-1}|_{L^\infty}\big).
$$
Furthermore, given $F\in H^s(\xT^n)$, the mapping $\chi\mapsto\chi^\star F$ is continuous from $W^{1,\infty}$ to $H^s$.
\end{proposition}
\begin{proposition}\label{PLCRough}
Suppose $\rho>1$, $\chi:\mathbb{T}^n\mapsto\mathbb{T}^n$ is a $C^\rho_*$ diffeomorphism, and $F\in H^{s+1}(\mathbb{S}^1)$ with $r:=s-n/2>0$. Then the following para-linearization formula holds:
$$
F\circ\chi=\chi^\star F+T_{F'\circ\chi}\chi+\R_{A}(\chi)F,
$$
where 
$$
\big\|\R_{A}(\chi)F\big\|_{H^{\rho+\min(\rho,r)}}
\leq K_s\big(|\chi|_{C^\rho_*},|\chi^{-1}|_{C^1}\big)\|F\|_{H^{s+1}}.
$$
Furthermore, given $F\in H^s(\xT^n)$ and $\varepsilon>0$, the mapping $\chi\mapsto\R_{A}(\chi)F$ is continuous from $C^{\rho}_*$ to $H^{s-\varepsilon}$.
\end{proposition}
\begin{remark}
In fact, continuous dependence of para-composition on the diffeomorphism $\chi$ was not explicitly proved in the literature cited above. However, one can directly read it off from the proof with a mere application of the dominated convergence theorem.
\end{remark}
We can give a refined proof of sovability of (\ref{ConjRot0}). We still assume in a priori that $u\in H^s$, where $s\geq\sigma+1.5+\varepsilon$, and $f\in H^{s+1}$. We then pick $r=\rho=s-1/2$. By Proposition \ref{PLCRough}, with $\chi=\Id+u$, we have the para-linearization formula for $\mathscr{F}(f,U)$:
\begin{equation}\label{PLF(f,U)1}
\begin{aligned}
\mathscr{F}(f,U)
&=\Delta_\alpha u-\chi^\star f-T_{f'\circ\chi}\chi
-\R_A(\chi)f+\lambda\\
&=\Delta_\alpha u-T_{\Delta_\alpha u'/(1+u')}u-\chi^\star f
+T_{[\mathscr{F}(f,U)]'/(1+u')}u
-\R_A(\chi)f+\lambda\\
&=T_{(1+u'\circ\tau_\alpha)}\Delta_\alpha T_{1/(1+u')}u
-\chi^\star f
+T_{[\mathscr{F}(f,U)]'/(1+u')}u
-\R_A(\chi)f+R_1(u)+\lambda.
\end{aligned}
\end{equation}
Here the computations are almost the same as those in (\ref{PLF(f,U)}): the remainder $R_1(u)\in H^{s+r-1}\subset H^{s+\sigma+\eps}$ and vanishes bilinearly when $H^s\ni u\mapsto 0$, while the para-composition remainder $\R_A(\chi)f\in H^{2s-1}$ by Proposition \ref{PLCRough}.

We then consider the \emph{para-homological equation} for the unknown $U=(u,\lambda)\in H^s\times\mathbb{R}$:
\begin{equation}\label{RotParaHom1}
\begin{aligned}
T_{(1+u'\circ\tau_\alpha)}\Delta_\alpha T_{1/(1+u')}u
&=\chi^\star f+\R_A(\chi)f-R_1(u)-\lambda.
\end{aligned}
\end{equation}
Equation (\ref{RotParaHom1}) is in fact just (\ref{RotParaHom}) in a delicate form. The para-inverse equation (\ref{RotParaInv}) then takes an equivalent form
\begin{equation}\label{RotParaInv1}
\begin{aligned}
u
&=T_{1/(1+u')}^{-1}\Delta_\alpha^{-1}T_{(1+u'\circ\tau_\alpha)}^{-1}\big(\chi^\star f+\R_A(\chi)f-R_1(u)-\lambda\big).
\end{aligned}
\end{equation}
Here the $\lambda$ is still uniquely determined to balance out the mean value, making $\Delta_\alpha^{-1}$ applicable.

Due to Proposition \ref{PCRough} and \ref{PLCRough}, we find that the right-hand-side of (\ref{RotParaInv1}) is a well-defined continuous mapping from $H^s$ to $H^{s+\eps}$. The para-differential remainder estimates further ensure that if $\|u\|_{H^s}\leq \delta$ and $|u'|_{L^\infty}\leq \rho$, then the $H^{s+\eps}$ norm of the right-hand-side of (\ref{RotParaInv}) is controlled by
$$
K(\delta,\rho)\|f\|_{H^{s+\sigma+\eps}}+C\delta^2,
$$
where $K$ is an increasing function. Thus if $\|f\|_{H^{s+\sigma+\eps}}$ is sufficiently small, the right-hand-side of (\ref{RotParaInv1}) will be a continuous mapping from the closed ball $\bar B(0,\delta)\subset H^s$ to itself, with compact range. By the Schauder fixed point theorem, such a mapping has a fixed point. The rest of the argument will then be identical with Subsection \ref{ParaDiffIllu}.

Sharing the same spirit with Subsection \ref{ParaDiffIllu}, the para-composition approach just presented obviously requires less regularity for $f$ when $\sigma\leq2.5$, at the price of being technically more complicated. To keep our narration as simple as possible, we do not elaborate this approach within this paper. See Appendix \ref{AppA} for more discussion.

\subsection{Heuristics about Generality}\label{HeuGen}
It turns out that the aforementioned method is valuable not only in one-dimensional scenarios. In fact, it encompasses all the features of general conjugacy problems of interest in dynamical systems. Roughly speaking, this is the para-differential calculus version of Zehnder's heuristics in Section 5 of \cite{Zehnder1975}. We refer the reader to Zehnder's paper for his original argument, and present here the para-differential version.

The formalism of conjugacy problems is as follows. Given an open set $\mathfrak{B}$ in a Banach space and an infinite dimensional Lie group $\mathfrak{G}$, suppose there is a differentiable group action (``dynamical system")
$$
\Phi:\mathfrak{B}\times\mathfrak{G}\mapsto\mathfrak{B}.
$$
Then $\Phi$ is subjected to the algebraic relation
\begin{equation}\label{Homomorphism}
\Phi(f,\Id)=f,
\quad
\Phi(f,g_2\circ g_1)=\Phi\big(\Phi(f,g_2),g_1\big).
\end{equation}
This is simply a result of definition of group homomorphism. 

Now suppose $\mathfrak{N}\subset\mathfrak{B}$ is a closed linear submanifold, usually being the \emph{normal forms} of the dynamical system. The conjugacy problem then reads: \emph{is $\mathfrak{N}$ structurally stable in this dynamical system?} That is, if $f\in\mathfrak{B}$ is close to $\mathfrak{N}$, is it possible to find an element $g\in\mathfrak{G}$ close to the identity, so that $\Phi(f,g)\in\mathfrak{N}$?

To answer this question, introduce the unknwon $u=(\gamma,\mathfrak{n})$, where $\gamma$ belongs to the Lie algebra of $\mathfrak{G}$ and $\mathfrak{n}\in\mathfrak{N}$. Suppose $\mathcal{E}$ is a local chart\footnote{This local chart usually is not the exponential map, since the exponential map is not surjective for an infinite dimensional Lie group. See for example, Section 42 of \cite{KM1997}.} near the identity of the group $\mathfrak{G}$. Define a mapping
$$
\mathscr{F}(f,u)=\Phi(f,\mathcal{E}(\gamma))-\mathfrak{n}.
$$
The conjugacy problem is then equivalent to finding zeroes of $\mathscr{F}$. Suppose that an approximate solution $u_0=(0,\mathfrak{n}_0)$ is already found, that is $\mathscr{F}(f,u_0)$ is close to 0. We then aim to look for $u$ close to 0 such that $\mathscr{F}(f,u)=0$

Introduce the map $\chi_\gamma$ on the Lie algebra of $\mathfrak{G}$ by $\mathcal{E}(\chi_\gamma(\gamma_1))=\mathcal{E}(\gamma)\circ\mathcal{E}(\gamma_1)$. Then (\ref{Homomorphism}) becomes
$$
\Phi\big(f,\chi_\gamma(\gamma_1)\big)
=\mathscr{F}\big(\Phi(f,\gamma),\gamma_1\big).
$$
Denote by $v=(\psi,\mathfrak{m})$ the increment of $u=(\gamma,\mathfrak{n})$, and set $L_\gamma v=(D\chi_\gamma(0)\psi,\mathfrak{m})$. Differentiating with respect to $u$, evaluating at $\gamma_1=0$, we obtain
$$
D_u\mathscr{F}(f,u)L_\gamma v
=D_u\mathscr{F}\big(\Phi(f,\gamma),(0,\mathfrak{n})\big)v.
$$ 
Thus by Taylor's formula,
$$\begin{aligned}
D_u\mathscr{F}(f,u)L_\gamma v
&=D_u\mathscr{F}\big(\mathfrak{n},(0,\mathfrak{n})\big)v
+\left(D_u\mathscr{F}\big(\Phi(f,\gamma),(0,\mathfrak{n})\big)
-D_u\mathscr{F}\big(\mathfrak{n},(0,\mathfrak{n})\big)\right)v\\
&=D_u\mathscr{F}\big(\mathfrak{n},(0,\mathfrak{n})\big)v
+\bm{B}\big(\mathscr{F}(f,u),v\big),
\end{aligned}
$$
where $\bm{B}$ vanishes bilinearly near (0,0). Equivalently,
\begin{equation}\label{DiffIdentity}
\begin{aligned}
D_u\mathscr{F}(f,u)v
&=D_u\mathscr{F}\big(\mathfrak{n},(0,\mathfrak{n})\big)L_\gamma^{-1}v
+\bm{B}\big(\mathscr{F}(f,u),L_\gamma^{-1}v\big).
\end{aligned}
\end{equation}

The idea of para-differential approach to the conjugacy equation $\mathscr{F}(f,u)=0$ is then based on the algebraic identity (\ref{DiffIdentity}). Suppose the Banach space $\mathfrak{B}$ and Lie group $\mathfrak{G}$ are realized as functions and (finite-dimensional) diffeomorphisms, and the mapping $\mathscr{F}(f,u)$ is a classical nonlinear differential operator acting on $u$, in the sense that it is a function of $(u,\partial u,\partial^2u,\cdots)$. Then the {para-linearization theorem of Bony} \cite{Bony1981} should imply
\begin{equation}\label{ParaLinF(f,u)}
\begin{aligned}
\mathscr{F}(f,u)
&=\mathscr{F}(f,u_0)+T_{D_u\mathscr{F}(f,u)}(u-u_0)
+\bm{R}_1(u-u_0)\\
&=T_{D_u\mathscr{F}\big(\mathfrak{n},(0,\mathfrak{n})\big)}T_{L_\gamma^{-1}}(u-u_0)
+\mathscr{F}(f,u_0)+\bm{R}_2(u-u_0)
+\bm{B}\big(T_{\mathscr{F}(f,u)},T_{L_\gamma^{-1}}(u-u_0)\big).
\end{aligned}
\end{equation} 
Here $u=(\gamma,\mathfrak{n})\simeq(0,\mathfrak{n}_0)$, and $\bm{R}_1,\bm{R}_2$ are smoother remainders produced by para-linearization. The somewhat abused notation $\bm{B}\big(T_{\mathscr{F}(f,u)},T_{L_\gamma^{-1}}(u-u_0)\big)$ stands for the replacement of ``product" $\bm{B}\big({\mathscr{F}(f,u)},L_\gamma^{-1}(u-u_0)\big)$ by the corresponding ``para-product" $\bm{B}\big(T_{\mathscr{F}(f,u)},T_{L_\gamma^{-1}}(u-u_0)\big)$, should $\bm{B}$ involve only usual functional operations.

The following step is to make sure that the \emph{para-homological equation}
\begin{equation}\label{ParaHomHeu}
T_{D_u\mathscr{F}\big(\mathfrak{n},(0,\mathfrak{n})\big)}T_{L_\gamma^{-1}}(u-u_0)
+\mathscr{F}(f,u_0)+\bm{R}_2(u-u_0)=0
\end{equation}
is solvable. We will actually consider its equivalent fixed point form, refered as the \emph{para-inverse equation}:
$$
(u-u_0)
=-T_{L_\gamma^{-1}}^{-1}T_{D_u\mathscr{F}\big(\mathfrak{n},(0,\mathfrak{n})\big)}^{-1}\big(\mathscr{F}(f,u_0)+\bm{R}_2(u-u_0)\big),
$$
and make sure that it is solvable. 

For most conjugacy problems, this really is the case, \emph{even if inverting $D_u\mathscr{F}\big(\mathfrak{n},(0,\mathfrak{n})\big)$ causes a loss of regularity due to small denominator}. In fact, that loss of regularity can be compensated by the additional smoothness of the para-differential remainder $\bm{R}_2(u-u_0)$. Thus, if $\mathscr{F}(f,u_0)$ is sufficiently smooth, the right-hand-side of the fixed point equation defines a mapping with \emph{no} loss of regularity (e.g. from $H^s$ to $H^s$). We note that this is the para-differential version of Zehnder's ``finding approximate right inverse". Consequently, assuming $\mathscr{F}(f,u_0)$ is close enough to 0, a standard fixed point argument should produce a solution $u\simeq u_0$ of (\ref{ParaHomHeu}).

Once a solution $u$ to the para-homological equation (\ref{ParaHomHeu}) is found, the first three terms in the right-hand-side of (\ref{ParaLinF(f,u)}) are then cancelled, and the para-linearization formula then becomes
$$
\mathscr{F}(f,u)
=\bm{B}\big(T_{\mathscr{F}(f,u)},L_\gamma^{-1}(u-u_0)\big)
\approx T_{\mathscr{F}(f,u)}L_\gamma^{-1}(u-u_0).
$$
Now since a para-differential operator $T_a$ depends only on very mild regularity of the symbol $a$, this equation must imply $\mathscr{F}(f,u)=0$ by a Neumann series argument, provided that $u$ is close to $u_0$. 

\emph{To conclude, solution of the para-homological equation in fact solves the original conjufacy problem.}

Of course, the steps outlined above are merely heuristics. In practice, each of the assumptions they rely upon must be validated in a case-by-case fashion. Nevertheless, given that these intuitive deductions pertain to very general scenarios, it is reasonable to expect that they adequately cover several well-studied conjugacy problems. We point out, omitting most details, that the following model problems all fit into this formalism: conjugation of standard map (see e.g. Section 2 of \cite{Llave2001}), structural stability of parallel vector field on torus (see e.g. \cite{Poschel2011}), existence of invariant torus of symplectic map (see e.g. \cite{DGJV2005}), and conjugation of Hamiltonian vector fields. We will elaborate on the last one in Section \ref{InvTor}.

The idea of replacing a linearized equation by its para-differential version dates back to H\"{o}rmander's paper \cite{Hormander1990}, where H\"{o}rmander proposed the terminology \emph{para-inverse}. H\"{o}rmander only considered nonlinear mappings with \emph{exactly} invertible linearized operators, instead of \emph{approximately} invertible ones. Nevertheless, it is still appropriate to attribute the core idea of the aforementioned heuristics to \cite{Hormander1990} (and \cite{Zehnder1975}).

\section{Quantitative Para-product and Para-linearization Estimates}\label{Quant}

In this section, we present quantitative estimates concerning para-products and para-linearization. They are refinements of standard contents in modern harmonic analysis; see for example, \cite{Bony1981,Meyer1980,meyer1980remarques,Meyer1981}, Chapter 9-10 of \cite{Hormander1997}, or Chapter 3-4 of \cite{Met2008}. From a harmonic analysis point of view, para-differential operators fall into special subclasses of pseudo-differential operators of ``forbidden" type (1,1), but still enjoy a symbolic calculus as usual (1,0) pseudo-differential operators do. However, to keep prerequisites minimal, we attempt to give neither comprehensive exposition of general (1,1) pseudo-differential operators nor the symbolic calculus, but confine ourselves with what is necessary. 

\subsection{Meyer Multipliers}
Before providing the quantitative para-differential operator estimates, let us first state some auxiliary results relating to \emph{Meyer multipliers}. Formally, given a sequence $\{m_j(x)\}_{j\geq0}$ of functions on $\xT^n$, we call the linear operator
$$
f(x)\mapsto\sum_{j\geq0} m_j(x)\cdot(\Delta_jf)(x)
$$
the Meyer multiplier associated to $\{m_j\}_{j\geq0}$. Here the $\Delta_j$'s are the Littlewood-Paley building block operators introduced in (\ref{LP_Decomp}). The fundamental theorem for such a linear operator is the following:
\begin{proposition}[Meyer multiplier]\label{Meyer}
Let $s,r\in\mathbb{R}$, and suppose that $s+r>0$. Let $N_{s+r}$ be the smallest integer such that $N_{s+r}>s+r$. Suppose $(m_j)_{j\geq0}$ is a sequence in $L^\infty(\xT^n)$, such that for multi-indices $\alpha$ with $|\alpha|\leq N_{s+r}$, there holds
$$
\big|\partial^\alpha m_j\big|_{L^\infty} 
\leq M_\alpha 2^{j(|\alpha|-r)}.
$$
Then the linear operator
$$
f\mapsto\sum_{j=0}^\infty m_j\Delta_j f
$$
is bounded from $H^s(\xT^n)$ to $H^{s+r}(\xT^n)$. The norm of this operator is estimated by
$$
C_s\sum_{\beta:|\beta|\leq N_s}M_\beta.
$$
\end{proposition}

We need a lemma concerning Sobolev regularity of series to aid the proof of Proposition \ref{Meyer}. It is of some independent interest per se.
\begin{lemma}\label{SummationLem}
Let $s>0$, let $N_s$ be the smallest integer such that $N_s>s$. There is a constant $C_s$ with the following properties. If $\{f_k\}_{k\geq0}$ is a sequence in $H^{N_s}(\xT^n)$, such that for any multi-index $\alpha$ with $|\alpha|\leq N_s$, there holds
$$
\|\partial^\alpha f_k\|_{L^2}\leq 2^{k(|\alpha|-s)}c_k,\quad (c_k)\in \ell^2(\mathbb{N}),
$$
then $\sum_k f_k\in H^s(\xR^n)$ and 
$$
\left\|\sum_{k\geq0}f_k\right\|_{H^s}^2
\leq C_s\sum_{k\geq0}c_k^2.
$$
\end{lemma}
\begin{proof}
If $j\leq k$, then by taking $\alpha=0$, obviously
$$
\|\Delta_jf_k\|_{L^2}\leq\|f_k\|_{L^2}\leq 2^{-ks}c_k.
$$
If $j>k$, then by Bernstein's inequality,
$$
\|\Delta_jf_k\|_{L^2}
\leq C_s2^{-N_sj}\sum_{|\alpha|=N_s}\|\partial^\alpha f_k\|_{L^2}
\leq C_s2^{-ks}2^{N_s(k-j)}c_k.
$$
Thus, for $f=\sum_kf_k$, we have
\begin{equation}\label{Block}
2^{js}\|\Delta_j f\|_{L^2}
\leq C_s\sum_{k:k<j}2^{(N_s-s)(k-j)}c_k
+\sum_{k:k\geq j}2^{(j-k)s}c_k.
\end{equation}
Given that $s>0$ and $N_s>s$, using an elementary double summation argument, the right-hand-side form an $\ell^2$ sequence (with index $j$), with norm controlled by $C_s\|(c_k)\|_{\ell^2}$. Thus $f\in H^s$.
\end{proof}

\begin{proof}[Proof of Proposition \ref{Meyer}]
In fact, for $|\alpha|\leq N_{s+r}$, each summand $m_j\Delta_j f$ satisfies
$$
\begin{aligned}
\big\|\partial^\alpha(m_j\Delta_jf)\big\|_{L^2}
&\leq C_\alpha \sum_{\beta:\beta\leq\alpha}
\big|\partial^\beta m_j\big|_{L^\infty}\big\|\partial^{\alpha-\beta}\Delta_jf\big\|_{L^2} \\
&\leq C_\alpha 2^{j(|\alpha|-r-s)}\left(\sum_{\beta:\beta\leq\alpha}M_\beta\right)2^{js}\|\Delta_jf\|_{L^2}.
\end{aligned}
$$
We can then directly apply Lemma \ref{SummationLem} to conclude.
\end{proof}

\begin{corollary}\label{SpecCondition}
If the sequence $\{f_k\}$ in Lemma \ref{SummationLem} in addition satisfies the spectral condition
$$
\supp\hat{f}_k\subset\{|\xi|\geq \gamma 2^k\}
$$
for some $\gamma>0$, then the restriction $s>0$ in Lemma \ref{SummationLem} can be relaxed to any real index $s$:
$$
\left\|\sum_{k\geq0}f_k\right\|_{H^s}^2
\leq C_{s,\gamma}\sum_{k\geq0}c_k^2.
$$
\end{corollary}
\begin{proof}
In fact, if the $f_k$'s satisfy the spectral condition, then the second sum in (\ref{Block}) becomes a finite one, with number of summands approximately equal to $\log \gamma^{-1}$. Thus the sequence $\sum_{k:k\geq j}2^{(j-k)s}c_k$ automatically forms an $\ell^2$ sequence (with index $j$).
\end{proof}

\begin{remark}\label{DiffMeyer}
Suppose the multipliers $m_j=m_j(\lambda)$ in Proposition \ref{Meyer} depend on a parameter $\lambda$ varying in some Banach space, such that
$$
\max_{\alpha:|\alpha|\leq N_{s+r}}\sup_j2^{-j(|\alpha|-r)}\big|\partial^\alpha D_\lambda m_j(\lambda)\big|_{L^\infty} <\infty.
$$
Then we can simply repeat the proof of Proposition \ref{Meyer} to conclude the following: the operator-valued function
$$
\lambda\mapsto\sum_{j=0}^\infty m_j(\lambda)\Delta_j
\in\mathcal{L}(H^s,H^{s+r})
$$
is differentiable in $\lambda$ under the norm topology of $\mathcal{L}(H^s,H^{s+r})$, whose differential is simply
$$
\sum_{j=0}^\infty D_\lambda m_j(\lambda)\Delta_j.
$$
\end{remark}

\subsection{Quantitative Estimates for Para-product Operators}
With the aid of Meyer multiplier estimates, we are now able to provide a quantitative proof of boundedness of para-product operators. We will first list several properties related to Littlewood-Paley characterization of the Zygmund space. The proof is quite elementary using the definition of Zygmund spaces, as in (\ref{Zygmund}).

\begin{proposition}\label{ZygmundDecay}
Suppose $r>0$ and $f\in C^r_*$. Then for any multi-index $\alpha$, there holds
$$
\big|\partial^\alpha S_jf\big|_{L^\infty}
\lesssim_{r,\alpha}
\left\{
\begin{aligned}
    & 2^{j(|\alpha|-r)_+}|f|_{C^r_*} & \quad |\alpha|\neq r \\
    & j|f|_{C^r_*} & \quad |\alpha|=r
\end{aligned}
\right.
$$
where $s_+=\max(s,0)$. Furthremore, there holds
$$
|f-S_jf|\lesssim_r 2^{-jr}|f|_{C^r_*}. 
$$
\end{proposition}

We are now at the place to state and prove the continuity of para-product operators, which appeared as Proposition \ref{ContRough}.

\begin{proposition}\label{ContPM}
For all $a\in L^\infty(\xT^n)$ and all $s\in\xR$, the para-product $T_a$ is a bounded linear operator from $H^s(\xT^n)$ to $H^s(\xT^n)$:
$$
\lA T_a\rA_{\mathcal{L}(H^s,H^s)}\le C_s|a|_{L^\infty}.
$$
\end{proposition}
\begin{proof}
In fact, in the defining equality (\ref{ParaMult}), each summand has Fourier support in the annulus $\{0.25\cdot 2^j \leq|\xi|\leq 2.25\cdot 2^j\}$. We then simply notice that the Meyer multipliers $m_j=S_{j-3}a$ satisfy
$$
\big|\partial^\alpha m_j\big|_{L^\infty} 
\leq C_\alpha 2^{j|\alpha|}|a|_{L^\infty}
$$
by Proposition \ref{ZygmundDecay}. Thus we can repeat the proof of Corollary \ref{SpecCondition} and conclude.
\end{proof}

The non-trivial fact about para-product operators is that they behave as if they were genuine multiplication operators: the multiplication $T_aT_b$ is the same as $T_{ab}$ modulo smoothing operators. Formally, this means that $a\mapsto T_a$ is an algebra homomorphism from the algebra of functions to the algebra of bounded operators (modulo smoothing ones). We provide the precise version of Proposition \ref{APRough}:

\begin{proposition}\label{ContCM}
Fix $r>0$, and suppose $a,b\in C^r_*$. Then
$$
\R_{\CM}(a,b)=T_a\circ T_b-T_{ab}
$$
is a bounded linear operator from $H^s(\xT^n)$ to $H^{s+r}(\xT^n)$ for all $s\in\mathbb{R}$, where CM is the abbreviation for \emph{composition de para-multiplication}:
$$
\lA \R_{\CM}(a,b)\rA_{\mathcal{L}(H^s, H^{s+r})}\le 
C_{s,r}|a|_{C^r_*}|b|_{C^r_*}.
$$
In other words, $\R_{\CM}:C^r_* \times C^r_* \mapsto \mathcal{L}(H^s, H^{s+r})$ is a bounded bilinear mapping.
\end{proposition}
\begin{proof}
The proof here is exactly the same as Theorem 2.3 of \cite{Bony1981}. Suppose $a,b\in C^r_*$ and $u\in H^s$ are given. We set $v=T_bu$, $v_q=S_{q-3}b\cdot \Delta_qu$. By Corollary \ref{SpecCondition} and the definition of Zygmund space norm,
$$
R_1v:=T_av-\sum_{p}S_{p-5}a\cdot \Delta_pv
=\sum_{p}(\Delta_{p-4}a+\Delta_{p-3}a)\Delta_pv
$$
satisfies $\|R_1v\|_{H^{s+r}}\leq C_{s,r}|a|_{C^r_*}|b|_{C^r_*}\|v\|_{H^{s}}$, since each summand $(\Delta_{p-4}a+\Delta_{p-3}a)\Delta_pv$ has Fourier support contained in the annulus $\{0.25\cdot 2^p\leq|\xi|\leq 2.25\cdot 2^p\}$. Thus
$$
\begin{aligned}
T_aT_bu-\sum_{q}\sum_{p:p\leq q-5}\Delta_pa\cdot v_q
&=\sum_{q}\sum_{p:p\leq q-5}\Delta_pa\cdot \big(\Delta_qv-v_q\big)+R_1T_bu\\
&=\sum_{p\geq-5}\Delta_pa\sum_{q:p+5\leq q \leq p+7} \big(\Delta_qv-v_q\big)+R_1T_bu.
\end{aligned}
$$
The last step is because $\sum_q v_q=\sum_q \Delta_qv$ and $v_q$ has Fourier support contained in the annulus $\{0.25\cdot 2^q\leq|\xi|\leq 2.25\cdot 2^q\}$. Using Corollary \ref{SpecCondition} once again, we find
\begin{equation}\label{CM1}
\left\|T_aT_bu-\sum_{q}\sum_{p_1\leq q-5}\sum_{p_2\leq q-5}\Delta_{p_1}a\cdot\Delta_{p_2}b\cdot\Delta_qu\right\|_{H^{s+r}}
\leq C_{s,r}|a|_{C^r_*}|b|_{C^r_*}\|u\|_{H^s}.
\end{equation}

On the other hand, if we set
$$
c_q=\sum_{p_1\leq q-5}\sum_{p_2\leq q-5}\Delta_{p_1}a\cdot\Delta_{p_2}b,
$$
then 
\begin{equation}\label{CM2}
\begin{aligned}
|ab-c_q|_{L^\infty}
&\leq\sum_{p_1> q-5\text{ or }p_2> q-5}|\Delta_{p_1}a|\cdot|\Delta_{p_2}b|\\
&\leq C_r|a|_{C^r_*}|b|_{C^r_*}\sum_{p_1+p_2> q-5}2^{-(p_1+p_2)r}\\
&\leq C_r|a|_{C^r_*}|b|_{C^r_*}2^{-qr}.
\end{aligned}
\end{equation}
Noticing that
$$
T_{ab}u-\sum_q c_q\Delta_qu
=\sum_{q}\big(S_{q-3}(ab)-c_q\big)\cdot\Delta_qu,
$$
using Corollary \ref{SpecCondition} and (\ref{CM2}), we find
\begin{equation}\label{CM3}
\left\|T_aT_bu-\sum_{q}c_q\cdot\Delta_qu\right\|_{H^{s+r}}
\leq C_{s,r}|a|_{C^r_*}|b|_{C^r_*}\|u\|_{H^s}.
\end{equation}
Combining (\ref{CM1}) and (\ref{CM3}), the proof is complete.
\end{proof}

\subsection{Bony's Para-linearization Theorem}
In this subsection, we present a quantitative refinement of Bony's para-linearization theorem in \cite{Bony1981}, which asserts that given a nonlinear composition $F(u)$, the para-product $T_{F'(u)}u$ carries most of the irregularity, and
$$
F(u)-T_{F'(u)}u
$$
is more regular than $F(u)$. This is a far-reaching generalization of the para-product decomposition.

While Bony's theorem is well-established and extensively applied in modern nonlinear analysis, we still need to manage the \emph{size} of the smoothing remainder to effectively study the para-homological equation. We state the following refined version of Proposition \ref{PLRough}:
\begin{proposition}\label{Bony}
Fix $s,r>0$, and suppose $u\in (H^s\cap C^r_*)(\xT^n;\mathbb{R}^L)$. Let $N_{s+r}$ be the smallest integer $>s+r$. Suppose $F=F(x,z)\in C^{N_{s+r}+2}(\xT^n\times\mathbb{R}^L)$. Then there holds the following para-linearization formula:
$$
\begin{aligned}
F(x,u)-F(x,0)
&=T_{F'_z(x,u)}u
+\R_{\PL}\big(F(x,\cdot),u\big)u\\
&=:T_{F'_z(x,u)}u
+\R_{\PL;1}\big(F'_z(x,0)\big)u
+\R_{\PL;2}\big(F(x,\cdot),u\big)u
\in H^s+H^{s+r}.
\end{aligned}
$$
Here $\R_{\PL;1}\big(F'_z(x,0)\big):H^s\mapsto H^{s+r}$ and $\R_{\PL;2}\big(F(x,\cdot),u\big):H^s\mapsto H^{s+r}$ are bounded linear operators, so that for some increasing function $K_{r,s}$ depending only on $r,s$,
$$
\begin{aligned}
\left\|\R_{\PL;1}\big(F'_z(x,0)\big)\right\|_{\mathcal{L}(H^s,H^{s+r})}
&\leq C_{r,s}\big|F'_z(x,0)-\Avg F'_z(x,0)\big|_{C^{r}_*}\\
\left\|\R_{\PL;2}\big(F(x,\cdot),u\big)\right\|_{\mathcal{L}(H^s,H^{s+r})}
&\leq K_{r,s}\big(|u|_{C^r_*}\big)|F|_{C^{N_{s+r}+2}}|u|_{C^r_*}
\end{aligned}
$$
Moreover, the operator $\R_{\PL;2}\big(F(x,\cdot),u\big)\in\mathcal{L}(H^s,H^{s+r})$ depends continuously on $u\in H^s\cap C^r_*$ in the operator norm. 

If in addition, $F=F(x,z)$ is of better regularity $C^{N_{s+r}+3}(\xT^n\times\mathbb{R}^L)$ instead of merely $C^{N_{s+r}+2}(\xT^n\times\mathbb{R}^L)$, then both operators $T_{F'_z(x,u)}\in\mathcal{L}(H^s,H^{s})$ and $\R_{\PL;2}\big(F(x,\cdot),u\big)\in\mathcal{L}(H^s,H^{s+r})$ are continuously differentiable in $u\in H^s\cap C^{r}_*$ with respect to operator norms.
\end{proposition}
\begin{proof}
The proof is a refinement of the standard telescope series argument; see for example, Chapter 10 of \cite{Hormander1997}. We write
\begin{equation}\label{TeleF}
\begin{aligned}
F(x,u)-F(x,0)
&=F(x,S_0u)+\sum_{j=1}^\infty F(x,S_ju)-F(x,S_{j-1}u)\\
&=\left(\int_0^1F'_z(x,\tau \Delta_0u)\diff\!\tau\right)\Delta_0u
+\sum_{j=1}^\infty\left(\int_0^1 F'_z\big(x,S_{j-1}u+\tau\Delta_j u\big)\diff\!\tau\right)\Delta_ju.
\end{aligned}
\end{equation}
Defining the multipliers
$$
\begin{aligned}
m_j^1&=(1-S_{j-3})\big(F'_z(x,0)-\Avg F'_z(x,0)\big),\quad j\geq0\\
m_0^2&=\int_0^1 \Big[F'_z\big(x,\tau\Delta_0 u\big)-F'_z(x,0)\Big]\diff\!\tau
-S_{-3}\big(F'_z(x,u)-F'_z(x,0)\big),\\
m_j^2&=\int_0^1 \Big[F'_z\big(x,S_{j-1}u+\tau\Delta_j u\big)-F'_z(x,0)\Big]\diff\!\tau
-S_{j-3}\big(F'_z(x,u)-F'_z(x,0)\big), \quad j\geq1
\end{aligned}
$$
we can rewrite the telescope series (\ref{TeleF}) as
$$
F(x,u)-F(x,0)
=T_{F'_z(x,u)}u+\sum_{j=0}^\infty m_j^1\Delta_ju
+\sum_{j=0}^\infty m_j^2\Delta_ju.
$$
Let us then just define
$$
\R_{\PL;1}\big(F'_z(x,0)\big)u
:=\sum_{j=0}^\infty m_j^1\Delta_ju,
\quad
\R_{\PL;2}\big(F(x,\cdot),u\big)u
:=\sum_{j=0}^\infty m_j^2\Delta_ju.
$$
In view of Proposition \ref{Meyer}, it suffices to estimate the two sets of Meyer multipliers $\{m_j^1\}$ and $\{m_j^2\}$.

For the multipliers $\{m_j^1\}$, we estimate them by the simple decay property in Proposition \ref{ZygmundDecay} for the Zygmund class $C^r_*$. Since $F'_z(x,0)\in C^{N_{s+r}+1}$ by assumption, we obtain by Proposition \ref{ZygmundDecay} that
$$
|\partial^\alpha m_j^1|_{L^\infty}
\leq C_{s,\alpha} 2^{j(|\alpha|-r)}\big|F'_z(x,0)-\Avg F'_z(x,0)\big|_{C^{r}},
\quad |\alpha|\leq N_{s+r}.
$$
Proposition \ref{Meyer} then ensures the estimate for the operator $\R_{\PL;1}\big(F'_z(x,0)\big)$.

The estimate for multipliers $\{m_j^2\}$ follows by the same idea but is technically more involved. In view of Proposition \ref{Meyer}, it suffices to show that
\begin{equation}\label{m_j^2}
|\partial^\alpha m_j^2|_{L^\infty}
\leq K_{s,\alpha}\big(|u|_{C^r_*}\big) 2^{j(|\alpha|-r)}|F|_{C^{N_{s+r}+2}}|u|_{C^r_*},
\quad |\alpha|\leq N_{s+r}.
\end{equation}

To obtain (\ref{m_j^2}) with $\alpha=0$, we notice that for $j\geq1$,
\begin{equation}\label{0_m_j^2}
\begin{aligned}
m_j^2&=\int_0^1 \Big[\left(F'_z\big(x,S_{j-1}u+\tau\Delta_j u\big)-F'_z(x,u)\right)\Big]\diff\!\tau
+(1-S_{j-3})\big(F'_z(x,u)-F'_z(x,0)\big) \\
&=\int_{[0,1]^2}F''_{zz}\big(x,\tau_2(S_{j-1}u+\tau_1\Delta_j u)\big)\big((1-S_{j-1})u+\tau_1\Delta_ju\big)\diff\!\tau_1\diff\!\tau_2 \\
&\quad+(1-S_{j-3})\int_0^1F''_{zz}(x,\tau u)u\diff\!\tau.
\end{aligned}
\end{equation}
The first term in the right-hand-side of (\ref{0_m_j^2}) is directly seen (by Proposition \ref{ZygmundDecay}) to be controlled by $2^{-jr}|F|_{C^{2}}|u|_{C^r_*}$. For the second term, we simply notice that $F\in C^{N_{s+r}+2}\subset C^{r+2}$, so
$$
\big|F''_{zz}(x,\tau u)u\big|_{C^r_*}
\leq K_{s}\big(|u|_{C^r_*}\big) |F|_{C^{r+2}}|u|_{C^r_*},
$$
which then implies the desired estimate by implementing Proposition \ref{ZygmundDecay} once again. The estiamte for $m_0^2$ is similar. 

To obtain (\ref{m_j^2}) with $|\alpha|=N_{s+r}$, we rewrite $m_j^2$ in a different manner: for $j\geq1$,
\begin{equation}\label{N_sm_j^2}
\begin{aligned}
m_j^2
&=\int_{[0,1]^2}F''_{zz}\big(x,\tau_2(S_{j-1}u+\tau_1\Delta_j u)\big)\big(S_{j-1}u+\tau_1\Delta_j u\big)\diff\!\tau_1\diff\!\tau_2 
-S_{j-3}\int_0^1F''_{zz}(x,\tau u)u\,\diff\!\tau.
\end{aligned}
\end{equation}
We then use $\big|F''_{zz}(x,\tau u)u\big|_{C^r_*}
\leq K_{s}\big(|u|_{C^r_*}\big) |F|_{C^{r+2}}|u|_{C^r_*}$ again: for $|\alpha|=N_{s+r}$, by Proposition \ref{ZygmundDecay}, the second term in (\ref{N_sm_j^2}) satisfies
$$
\left|\partial^\alpha S_{j-3}\int_0^1F''_{zz}(x,\tau u)u\,\diff\!\tau\right|
\leq 2^{j(N_{s+r}-r)}K_{s}\big(|u|_{C^r_*}\big) |F|_{C^{r+2}}|u|_{C^r_*}.
$$
As for the first term in (\ref{N_sm_j^2}), by the Faà di Bruno formula, the partial derivative
$$
\partial^\alpha\Big[F''_{zz}\big(x,\tau_2(S_{j-1}u+\tau_1\Delta_j u)\big)\big(S_{j-1}u+\tau_1\Delta_j u\big)\Big]
$$
is a linear combination of terms of the form
\begin{equation}\label{FDB_Summand}
(\tau_2D_z)^{2+I}F\big(x,\tau_2(S_{j-1}u+\tau_1\Delta_j u)\big)\cdot
\left[\partial^{\beta_1}\big(S_{j-1}u+\tau_1\Delta_j u\big)\right]^{i_1}
\cdots
\left[\partial^{\beta_l}\big(S_{j-1}u+\tau_1\Delta_j u\big)\right]^{i_l},
\end{equation}
where the scalar indices $i_1,\cdots,i_l,I\leq N_{s+r}$ and multi-indices $\beta_1,\cdots,\beta_l$ satisfy
$$
\begin{aligned}
N_{s+r}=i_1|\beta_1|+\cdots+i_l|\beta_l|,
\quad
I=i_1+\cdots+i_l.
\end{aligned}
$$
By the Littlewood-Paley characterization of Zygmund functions, such a term has an upper bound
\begin{equation}\label{FDB_Esti}
C_s|F|_{C^{N_{s+r}+2}}
\prod_{p:\beta_p>r}\big(|u|_{C^r_*}2^j\big)^{i_p(|\beta_p|-r)}
\prod_{p:\beta_p\leq r}\big(|u|_{C^r_*}j\big)^{i_p}
\leq 
K_{s}\big(|u|_{C^r_*}\big)
|F|_{C^{N_{s+r}+2}}|u|_{C^r_*}2^{j(N_{s+r}-r)}.
\end{equation}
Note that the expression vanishes at least linearly as $u\mapsto0$ since at least one $i_p\neq0$. The estimate for $\partial^\alpha m_0^2$ is similar. 

Therefore, (\ref{m_j^2}) holds for $|\alpha|=N_{s+r}$. By interpolation, it remains valid for those multi-indices $\alpha$ with length between $0$ and $N_{s+r}$. This proves the operator norm estimates of the para-differential remainders. 

To prove continuous dependence on $u\in H^s\cap C^r_*$ of the operator $\R_{\PL;2}\big(F(x,\cdot),u\big)\in\mathcal{L}(H^s,H^{s+r})$, in view of Proposition \ref{Meyer}, it suffices to show that when $u_1,u_2$ are close in $C^r_*$, then the Meyer multipliers $m_j^2$ corresponding to $u_1$ and $u_2$ will be close to each other, that is, for $|\alpha|\leq N_{s+r}$, the quantities
\begin{equation}\label{u_1-u_2}
2^{-j(|\alpha|-r)}\big|\partial^\alpha\big(m_j^2(u_1)-m_j^2(u_2)\big)\big|_{L^\infty}
\end{equation}
will be close to zero. This is easily verified if we take into account the assumption $F\in C^{N_{s+r}+2}$. For example, in the summand (\ref{FDB_Summand}), the worst scenario is when all the indices $i_p\equiv1$, so that (\ref{FDB_Summand}) becomes a ``highest order derivative":
$$
(\tau_2D_z)^{2+N_{s+r}}F\big(x,\tau_2(S_{j-1}u+\tau_1\Delta_j u)\big)\cdot
\left[D\big(S_{j-1}u+\tau_1\Delta_j u\big)\right]^{\otimes N_{s+r}}.
$$
When evaluated at $u_1,u_2\in C^r_*$ that are close in the $C^r_*$ topology, we find that
$$
\left[(\tau_2D_z)^{2+N_{s+r}}\Delta_{12}F\big(x,\tau_2(S_{j-1}u+\tau_1\Delta_j u)\big)\right]
$$
will be close to zero since $D_z^{2+N_{s+r}}F(x,z)$ is continuous, while using Proposition \ref{ZygmundDecay}, we estimate
$$
\left|\Delta_{12}\left[D\big(S_{j-1}u+\tau_1\Delta_j u\big)\right]^{\otimes N_{s+r}}\right|_{L^\infty}
$$
similarly as in (\ref{FDB_Esti}), with a factor $|\Delta_{12}u|_{C^r_*}$. In conclusion, we obtain that (\ref{u_1-u_2}) is close to zero when $u_1$ is close to $u_2$ in $C^r_*$.  

Finally, let us suppose $F$ has better regularity, namely $C^{N_{s+r}+3}(\xT^n\times\mathbb{R}^L)$. Continuous differentiability of the operator $\R_{\PL;2}\big(F(x,\cdot),u\big)$ in $u$ is a direct consequence of the equality
$$
(D_um_j^2)v=\int_0^1 \Big[F''_{zz}\big(x,S_{j-1}u+\tau\Delta_j u\big)(S_{j-1}+\tau\Delta_j)v\Big]\diff\!\tau
-S_{j-3}\big(F''_{zz}(x,u)v\big)
$$
and Remark \ref{DiffMeyer}. 
\end{proof}

\section{Existence of Invariant Torus}\label{InvTor}
In this section, we present detailed proof of Theorem \ref{Thm1}-\ref{Thm2} together with some possible extensions.

\subsection{Algebraic Set-up}
Let us recall that we are working with the phase space $\xT^n\times\xR^n$. We are interested in Hamiltonian functions $h(x,y)$ as in (\ref{h(x,y)}) and embeddings of $\xT^n$ into the phase space close to the ``flat" one (\ref{zeta_0})\footnote{Unfortunately, we have to use $\theta\in\xT^n$ to denote the variable on the ``abstract torus" since $x$ has already been occupied to label points in phase space.}. The torus $\{y=0\}$ is supposed to be ``approximately invariant" under the phase flow of Hamilton vector field $X_h$, and the restriction of $X_h$ is approximately the rotation generated by a constant frequency vector field $\omega\in\xR^n$.

The Diophantine condition (\ref{Dio}) of $\omega$ enables us to assert continuity of $\nabla_\omega ^{-1}$, defined as
$$
\nabla_\omega ^{-1}f(\theta)
:=\sum_{k\in\mathbb{Z}^n\setminus\{0\}}\frac{\hat f(k)e^{ik\cdot \theta}}{i(k\cdot\omega)},
\quad\text{if }\hat f(0)=0.
$$
In fact, it is immediate that
\begin{equation}\label{ContHs}
\big\|\nabla_\omega ^{-1}f\big\|_{H^s}
\leq \gamma\|f\|_{H^{s+\sigma}},
\quad
\text{if }\hat f(0)=0.
\end{equation}
Furthermore, the fundamental solution of $\nabla_\omega $ enjoys certain regularity:
\begin{lemma}\label{ContKer}
If $\omega\in\xR^n$ satisfies the Diophantine condition (\ref{Dio}), then for $\tau>\sigma+1$, the series
$$
\sum_{k\in\mathbb{Z}^n\setminus\{0\}}\frac{e^{ik\cdot \theta}}{i(k\cdot\omega)|k|^\tau}
$$
is absolutely convergent, thus sums to a continuous function.
\end{lemma}

We now study the conjugacy equations (\ref{ConjNondege}) and (\ref{ConjTran}), that is,
$$
X_h(u)-\nabla_\omega u=0
$$
and
$$
X_{h_\xi}(u)-\nabla_\omega u=0.
$$
To employ the general heuristics discussed in Subsection \ref{HeuGen},  we define a mapping
\begin{equation}\label{F(h,u)}
\mathscr{F}(h,u):=X_h(u)-\nabla_\omega u,
\quad \text{whose value we denote as }E.
\end{equation}
It is the ``error function" measuring whether $u$ is indeed an invariant torus of $X_h$ or not. The problem is then reduced to looking for zero $u\simeq \zeta_0$ of $\mathscr{F}(h,u)$ under the condition of Theorem \ref{Thm1}, or looking for $\xi\in\xR^n$ and zero $u\simeq \zeta_0$ of $\mathscr{F}(h_\xi,u)$ under the condition of Theorem \ref{Thm2}. 

We claim that in order to solve $\mathscr{F}(h,u)=0$, it suffices to solve a slighly ``softer" equation
\begin{equation}\label{EQCounter}
\mathscr{F}(h,u)+\begin{pmatrix}0 \\ \mu\end{pmatrix}=0
\end{equation}
for an auxiliary constant vector $\mu\in\xR^n$. We directly quote the following geometric lemma from \cite{BB2015}:
\begin{lemma}[\cite{BB2015}, Lemma 3]\label{CounterTerm}
There holds
$$
\mu=\Avg\Big((\partial u^y)^\T \mathscr{F}^x(h,u)-(\partial u^x)^\T \big(\mathscr{F}^y(h,u)-\mu\big)\Big).
$$
In particular, $\mathscr{F}(h,u)=\begin{pmatrix}0 \\ \mu\end{pmatrix}$ necessarily implies $\mu=0$.
\end{lemma}

In this subsection, our main task is to compute the linearization of $\mathscr{F}(h,u)$ at a given $u$, as in \cite{DGJV2005}. For convenience in notation, write
\begin{equation}\label{A[u]}
\begin{aligned}
A[u]:=(DX_h)(u)
=\begin{pmatrix}
D_x\nabla_yh(u) & D_y\nabla_yh(u) \\
-D_x\nabla_xh(u) & -D_y\nabla_xh(u)
\end{pmatrix}
\in\bm{M}_{2n\times 2n},
\end{aligned}
\end{equation}
so that
$$
D_u\mathscr{F}(h,u)v=A[u]v-\nabla_\omega v.
$$

We also write
$$
N[u]=\big((\partial u)^\T \cdot\partial u\big)^{-1}\in\bm{M}_{n\times n},
\quad
M[u]=\big(
\partial u \quad (J\partial u)\cdot N[u]
\big)\in\bm{M}_{2n\times2n},
$$
provided that $\partial u^\T \cdot\partial u$ is invertible. This is true when $u$ is $C^1$ close to $\zeta_0$, whence $N[u]=I_n+O\big(\partial(u-\zeta_0)\big)$. This furthermore ensures that 
$$M[u]= \begin{pmatrix}I_n & \\ & -I_n\end{pmatrix}+O\big(\partial(u-\zeta_0)\big),
\quad
M[u]^{-1}=\begin{pmatrix}I_n & \\ & -I_n\end{pmatrix}+O\big(\partial(u-\zeta_0)\big).
$$

We then introduce a function $L[u]$, reflecting the ``lack of being Lagrangian" for the submanifold $u(\xT^n)\subset \xT^n\times\xR^n$, denoted by
\begin{equation}\label{L[u]}
L[u]=(\partial u)^\T \cdot J\cdot\partial u
\in\bm{M}_{n\times n}.
\end{equation}
Then $L[u]$ is nothing but the matrix representation of the pull-back of the symplectic form on $\xT^n\times\xR^n$. The submanifold $u(\xT^n)\subset \xT^n\times\xR^n$ is Lagrangian if and only if $L[u]=0$. In particular, this is the case when $u$ does solve the conjugacy equation $\mathscr{F}(h,u)=0$.

The linearization formula for $\mathscr{F}(h,u)$ is summarized as the following lemma\footnote{The authors would like to thank Giovanni Forni and Nicolò Tedesco for pointing out a computational error in the original version of the paper and for subsequent discussions on the details of the correction.}:
\begin{lemma}\label{Lin_F(h,u)}
Suppose $u$ is close to the trivial embedding $\zeta_0$. Define the $n\times n$ matrix function
\begin{equation}\label{S[u]}
\begin{aligned}
S[u]
&=\big(I_n+(N[u]L[u])^2\big)^{-1}N[u]\cdot (\partial u)^\T\cdot\big[A[u],J\big]\cdot(\partial u)\cdot N[u].
\end{aligned}
\end{equation}
Then the linearization of $\mathscr{F}(h,u)$ with respect to $u$ satisfies
\begin{equation}\label{F_u(h,u)M}
D_u\mathscr{F}(h,u)\big(M[u]v\big)
=M[u]\begin{pmatrix}
0_n & S[u] \\
0_n & 0_n
\end{pmatrix}v
-M[u]\nabla_\omega v
+\bm{B}\big(\partial\mathscr{F}(h,u),L[u]\big)v.
\end{equation}
Here in (\ref{F_u(h,u)M}), the matrix function $\bm{B}(\partial E_1,E_2)$ is rational\footnote{We say that a matrix function $f(X)$ is rational in the argument $X$, if the entries of $f(X)$ are all rational functions of entries of $X$.} in $\partial E_1$ and $E_2$, vanishing linearly as $\partial E_1,E_2\to0$, with coefficients being rational functions of $A[u]$, $\partial u$ and $\partial^2u$. We may then re-write (\ref{F_u(h,u)M}) as the equivalent form
\begin{equation}\label{F_u(h,u)}
D_u\mathscr{F}(h,u)v
=M[u]\begin{pmatrix}
0_n & S[u] \\
0_n & 0_n
\end{pmatrix}M[u]^{-1}v
-M[u]\nabla_\omega \big(M[u]^{-1}v\big)
+\bm{B}\big(\partial\mathscr{F}(h,u),L[u]\big)M[u]^{-1}v,
\end{equation}
\end{lemma}

\begin{proof}
Throughout this proof, for simplicity in notation, we shall drop the $[u]$ (indicating dependence on $u$), write e.g. $A$ for $A[u]$, and denote $E:=\mathscr{F}(h,u)$. For a given $u$ and its increment $v$, we first write
$$
\begin{aligned}
D_u\mathscr{F}(h,u)(M[u]v)
&=AMv
-(\nabla_\omega M)v-M\nabla_\omega v\\
&=:
\big(C_1\quad C_2\big)v-M\nabla_\omega v
\end{aligned}
$$
Here we write the matrix $AM
-(\nabla_\omega M)\in \bm{M}_{2n\times 2n}$ into block form $\big(C_1\quad C_2\big)$. The task is to determine the blocks $C_1,C_2\in\bm{M}_{2n\times n}$.

It is fairly easy to see that $C_1=\partial E$. In fact, we may write $M[u]$ into block form as
$$
M
=\big(
\partial u \quad (J\partial u)N
\big)
=:\big(M_1\quad M_2\big).
$$
Then obviously
\begin{equation}\label{C1}
\begin{aligned}
C_1
&=AM_1-\nabla_\omega M_1\\
&=A\partial u-\nabla_\omega\partial u
=\partial\big(\mathscr{F}(h,u)\big)
=\partial E.
\end{aligned}
\end{equation}

The true difficulty is with the expression of $C_2=A(J\partial u)N-\nabla_\omega  \big((J\partial u)N\big)$. We try to find $n\times n$ matrices $Q,W$, such that
\begin{equation}\label{C2_temp}
C_2
=(\partial u)Q+(J\partial u)NW
=M_1Q+M_2W.
\end{equation}
In solving this matrix system, we take into account that $L=(\partial u)^\T J\partial u$ is supposed to be close to 0. Multiplying (\ref{C2_temp}) from left by $(\partial u)^\T J$, using in sequence the Hamiltonian nature $JA=-A^\T J$ and the symplectic nature $J^2=-I_{2n}$, we obtain
\begin{equation}\label{QW1}
\begin{aligned}
LQ-W
&=(\partial u)^\T JA(J\partial u)N-(\partial u)^\T J\nabla_\omega  \big((J\partial u)N\big)\\
&=(\partial u)^\T A^\T(\partial u) N
+(\partial u)^\T(\nabla_\omega\partial u) N
+(\partial u)^\T(\partial u)\nabla_\omega N\\
&=(\partial E)^\T(\partial u)N.
\end{aligned}  
\end{equation}
Here in the last step we differentiated the definition $(\partial u)^\T(\partial u) N=I_n$ to deduce
$$
(\partial u)^\T(\nabla_\omega\partial u) N
+(\partial u)^\T(\partial u)\nabla_\omega N
=\big[\nabla_\omega(\partial u)^\T\big](\partial u)N,
$$
and then use the last equality of (\ref{C1}).

On the other hand, multiplying (\ref{C2_temp}) from left by $N(\partial u)^\T$, we obtain
\begin{equation}\label{QW2}
\begin{aligned}
Q+NLNW
&=N(\partial u)^\T A\cdot(J\partial u)N
-N(\partial u)^\T\nabla_\omega  \big((J\partial u)N\big)\\
&=N(\partial u)^\T \Big[A,J\Big](\partial u)N
+N(\partial u)^\T JA(\partial u)N\\
&\quad-N(\partial u)^\T J(\nabla_\omega\partial u)N
-N(\partial u)^\T J(\partial u) \nabla_\omega N\\
&=N(\partial u)^\T \Big[A,J\Big](\partial u)N
+N(\partial u)^\T J(\partial E) N
-NL\nabla_\omega N.
\end{aligned}
\end{equation}
We can then directly solve from (\ref{QW1})(\ref{QW2}) the expression of $Q,W$:
\begin{equation}\label{QW}
\begin{aligned}
Q&=S
+\big(I_n+(NL)^2\big)^{-1}
\Big(NLN(\partial E)^\T(\partial u)N
+N(\partial u)^\T J(\partial E)N-NL\nabla_\omega N\Big)\\
W&=LQ-(\partial E)^\T(\partial u)N.
\end{aligned}
\end{equation}
Therefore, we immediately see that both $W-LS$ and $Q-S$ are (matrix) rational functions of $L$ and $\partial E$, vanishing linearly in $L$ and $\partial E$, with coefficients being rational functions of $\partial u$ and $\partial^2u$.

Combining (\ref{C1}) and (\ref{C2_temp}), we find that the matrix $AM-(\nabla_\omega M)$ has block-form expression
$$
\begin{aligned}
AM-(\nabla_\omega M)
&=\big(\partial E\quad M_1S\big)
+\Big(0_{2n\times n} \quad M_1(Q-S)+M_2W\Big)\\
&=M\begin{pmatrix}
0_n & S \\
0_n & 0_n
\end{pmatrix}
+\Big(\partial E \quad M_1(Q-S)+M_2W\Big).
\end{aligned}
$$
Therefore, the remainder $\bm{B}(E,L)=\big(\partial E \quad M_1(Q-S)+M_2W\big)$, which, by (\ref{QW}), is a matrix rational function of $L$ and $\partial E$, vanishing linearly in $L$ and $\partial E$, with coefficients being rational functions of $A$, $\partial u$ and $\partial^2u$.
\end{proof}

\begin{remark}
Lemma \ref{Lin_F(h,u)} is in fact equivalent to Lemma 20 of \cite{DGJV2005}. It also appeared in \cite{BB2015}. However, we are not able to find an explicit expression for $D_u\mathscr{F}(h,u)v$ in these references, which is the reason that we choose to write down our own proof. We observe that (\ref{F_u(h,u)}) plays the same role as (\ref{F_U(f,U)}).
\end{remark}

A key feature of the Lagrangian character $L[u]$ is that it is in proportion to the error $E=\mathscr{F}(h,u)$. In fact by Lemma 19 of \cite{DGJV2005} or Lemma 5 of \cite{BB2015}, if we define the matrix function
$$
L_1(u,E):=\partial
\big((\partial u)^\T\cdot JE\big)-\left[\partial\big((\partial u)^\T\cdot JE\big)^\T\big)\right]^\T,
$$
then there holds
\begin{equation}\label{Lack}
\nabla_\omega L[u]
=L_1(u,\mathscr{F}(h,u))
\quad\text{or equivalently}\quad
L[u]
=\nabla_\omega ^{-1}L_1(u,\mathscr{F}(h,u)).
\end{equation}
Note that due to the presence of $\partial$, the function $L_1(u,E)$ necessarily has average zero. In other words, the matrix function $\nabla_\omega L[u]$ is nothing but the matrix representation of the exterior differential of the 1-form represented by $(\partial u)^\T\cdot J\mathscr{F}(h,u)$. Of course, this is a refined version of the observation that an invariant torus is necessarily a Lagrangian submanifold. Therefore, the mapping $\bm{B}\big(\mathscr{F}(h,u),L[u]\big)$ in (\ref{F_u(h,u)}) vanishes linearly in $E$ if $u$ is approximately invariant:
\begin{corollary}\label{B(E,L)}
Suppose $\omega$ satisfies the Diophantine condition (\ref{Dio}), so that (\ref{ContHs}) holds. Suppose $u\in H^s$ for some $s>\max(\sigma+2,n/2+2)$, while $|u-\zeta_0|_{C^2}$ is sufficiently close to zero. Then for any constant shift $(0,\mu)^\T\in\xR^{n}\times\xR^n$, there holds
$$
\big\|\bm{B}\big(\partial\mathscr{F}(h,u),L[u]\big)\big\|_{H^{s-(\sigma+2)}}
\leq \gamma K_s\big(|h|_{C^{[s]+3}},\|u\|_{H^s}\big)\big\|\mathscr{F}(h,u)+(0,\mu)^\T\big\|_{H^{s-1}}.
$$
The function $K_s$ is increasing in both arguments and does not depend on the shift $(0,\mu)^\T\in\xR^{n}\times\xR^n$.
\end{corollary}
\begin{proof}
We have already seen that $\bm{B}\big(\partial\mathscr{F}(h,u),L[u]\big)$ is a rational function of $\partial\mathscr{F}(h,u)$ and $L[u]$, vanishing linearly as the arguments $\to0$, with coefficients being rational functions of $A[u],\partial u,\partial^2 u$. Since $|u-\zeta_0|_{C^2}$ is close to 0, we find that both $\partial\mathscr{F}(h,u)$ and $L[u]$ are $C^0$ close to 0, since they only involve $u$ and $\partial u$. Consequently e.g. the matrix $I_n+(N[u]L[u])^2$ is $H^{s-1}$ close to the identity matrix.

Obviously, shifting $\mathscr{F}(h,u)$ by a constant addendum does not affect the value of $\bm{B}\big(\partial\mathscr{F}(h,u),L[u]\big)$. Therefore, 
$$
\big\|\bm{B}\big(\partial\mathscr{F}(h,u),L[u]\big)\big\|_{H^{s-(\sigma+2)}}
\leq K_s\big(|h|_{C^{[s]+3}},\|u\|_{H^s}\big)\left(\big\|\mathscr{F}(h,u)+(0,\mu)^\T\big\|_{H^{s-1}}
+\|L[u]\|_{H^{s-(\sigma+2)}}\right).
$$
Now we make use of (\ref{Lack}). Note that for the matrix function 
$$
L_1(u,E)=\partial
\big((\partial u)^\T\cdot JE\big)-\left[\partial\big((\partial u)^\T\cdot JE\big)^\T\big)\right]^\T,
$$
shifting $E$ by a constant addendum $(0,\mu)^\T$ yields the same value of $L_1(u,E)$, by a direct computation (using $\partial_i\partial_ju=\partial_j\partial_iu$). Therefore, by (\ref{Lack}), we have
$$
\begin{aligned}
\|L[u]\|_{H^{s-(\sigma+2)}}
&\leq \gamma \big\|L_1\big(u,\mathscr{F}(h,u)+(0,\mu)^\T\big)\big\|_{H^{s-2}}\\
&\lesssim \gamma K_s\big(\|u\|_{H^s}\big)\big\|\mathscr{F}(h,u)+(0,\mu)^\T\big\|_{H^{s-1}}.
\end{aligned}
$$
This concludes the proof.
\end{proof}

\subsection{Deriving Para-homological Equation}
So far we have been mostly quoting and re-writing the algebraic findings reported in Section 8 of \cite{DGJV2005}. However, from this point, we will proceed quite differently from \cite{DGJV2005}, or any other known reference on KAM theory, e.g. \cite{BB2015} -- instead of a Newtonian algorithm involving approximate right inverse, we employ the idea of Subsection \ref{HeuGen} to directly solve the para-homological equation. We may treat Theorem \ref{Thm1} and \ref{Thm2} in a unified manner, since for \ref{Thm1} it suffices to require that the parameter $\xi$ must be zero.

So we restrict to a neighbourhood of the already known approximate solution $\zeta_0$. We notice that $\mathscr{F}(h_\xi,u)$ is a first order nonlinear partial differential operator acting on the unknown $(u,\xi)$, and in fact no differential in the constant $\xi$ is involved at all. Let us recall from (\ref{Error1}) and (\ref{Error2}) that $e_0=\mathscr{F}(h,\zeta_0)$. Then the para-linearization formula
\begin{equation}\label{PLF(h,U)}
\begin{aligned}
\mathscr{F}(h_\xi,u)+\begin{pmatrix} 0 \\ \mu\end{pmatrix}
&=e_0+
\begin{pmatrix} \xi \\ \mu\end{pmatrix}
+T_{D_u\mathscr{F}(h,u)}(u-\zeta_0)
+\R_{\PL}\big(X_h(\zeta_0+\cdot\,),u-\zeta_0\big)(u-\zeta_0)
\end{aligned}
\end{equation}
is valid. The expression $\bm{B}(E,L[u])$ involves only $\partial E$, so it does not change if $E$ is shifted by a constant vector. Thus, writing for simplicity
$$
E:=\mathscr{F}(h_\xi,u)+\begin{pmatrix} 0 \\ \mu\end{pmatrix},
$$
we actually have\footnote{Our notation (\ref{ParaMult}) is that $\Op^\PM(a)=T_a$, and we prefer to use $\Op^\PM$ here in case the expression of $a$ is very lengthy.}, by (\ref{F_u(h,u)}), the following equality:
$$
\begin{aligned}
T_{D_u\mathscr{F}(h,u)}(u-\zeta_0)
&=
\Op^\PM\left[M[u]\begin{pmatrix}
0_n & T_{S[u]} \\
0_n & 0_n
\end{pmatrix}M[u]^{-1}\right](u-\zeta_0)\\
&\quad-\Op^{\PM}\Big({M[u]}\big(\nabla_\omega M[u]^{-1}\big)+\nabla_\omega \Big)(u-\zeta_0)
+\Op^\PM\left(\bm{B}(E,L[u]) M[u]^{-1}\right)(u-\zeta_0).
\end{aligned}
$$
Thus the para-linearization formula (\ref{PLF(h,U)}) becomes
\begin{equation}\label{PLF(h,u)}
\begin{aligned}
E&=
e_0+\begin{pmatrix} \xi \\ \mu\end{pmatrix}
+T_{M[u]}\begin{pmatrix}
0_n & T_{S[u]} \\
0_n & 0_n
\end{pmatrix}T_{M[u]^{-1}}(u-\zeta_0)
-T_{M[u]}\nabla_\omega T_{M[u]^{-1}}(u-\zeta_0)\\
&\quad
+\Op^\PM\left(\bm{B}(E,L[u]) M[u]^{-1}\right)(u-\zeta_0)
+\R_{\CM}[u](u-\zeta_0)
+\R_{\PL}\big(X_h(\zeta_0+\cdot\,),u-\zeta_0\big)(u-\zeta_0).
\end{aligned}
\end{equation}
Here the remainder
\begin{equation}\label{RCM}
\begin{aligned}
\R_{\CM}[u]
&=\Op^\PM\left[M[u]\begin{pmatrix}
0_n & T_{S[u]} \\
0_n & 0_n
\end{pmatrix}M[u]^{-1}\right]
-T_{M[u]}\begin{pmatrix}
0_n & T_{S[u]} \\
0_n & 0_n
\end{pmatrix}T_{M[u]^{-1}}\\
&\quad-\Op^{\PM}\Big[{M[u]}\big(\nabla_\omega M[u]^{-1}\big)+\nabla_\omega \Big]
+T_{M[u]}\nabla_\omega T_{M[u]^{-1}}
\end{aligned}
\end{equation}
is produced in view of Proposition \ref{ContCM}. Note that we used the Leibniz rule $\partial T_au=T_{\partial a}u+T_a\partial u$.

We can now formulate the para-homological equation in the unknown $(u,\xi,\mu)$, where $\xi,\mu\in\xR^n$ are constant parameters to be fixed. The equation reads
\begin{equation}\label{ParaHomoHamilton}
\begin{aligned}
T_{M[u]}&\begin{pmatrix}
0_n & T_{S[u]} \\
0_n & 0_n
\end{pmatrix}T_{M[u]^{-1}}(u-\zeta_0)
-T_{M[u]}\nabla_\omega T_{M[u]^{-1}}(u-\zeta_0)
+\begin{pmatrix}
\xi \\ \mu
\end{pmatrix} \\
&=-e_0
-\R_{\CM}[u](u-\zeta_0)
-\R_{\PL}\big(X_h(\zeta_0+\cdot\,),u-\zeta_0\big)(u-\zeta_0)
\,.
\end{aligned}
\end{equation}
In other words, it aims to cancel out all terms in the right-hand-side of (\ref{PLF(h,u)}) but the one linear in $E$.

In order to solve the para-homological equation (\ref{ParaHomoHamilton}), we need a lemma concerning \emph{linear} para-homological equations, which is essentially the simple equation $\nabla_\omega v=f$.
\begin{lemma}\label{LinParaHomo}
Let $\omega$ be a Diophantine frequency vector satisfying (\ref{Dio}). Set 
$$
M_Q=\Bigg\{\begin{aligned}
\max\Big(\big|(\Avg Q)^{-1}\big|&,|Q|_{L^\infty}\Big)
&\quad\text{under the assumption of Theorem \ref{Thm1},}\\
&|Q|_{L^\infty}&\text{ under the assumption of Theorem \ref{Thm2}.}
\end{aligned}
$$
There are constants $\rho_1,\rho_2$ depending on $|h|_{C^3}$ and $M_Q$ with the following property. If 
$|e_0|_{C^1}\leq\rho_1$, and the embedding $u:\xT^n\mapsto\xT^n\times\xR^n$ is such that $|u-\zeta_0|_{C^1}\leq \rho_2$, then the linear para-homological equation in the unknown $(v,\xi,\mu)$
\begin{equation}\label{LinParaHomo'}
T_{M[u]}\begin{pmatrix}
0_n & T_{S[u]} \\
0_n & 0_n
\end{pmatrix}T_{M[u]^{-1}}v
-T_{M[u]}\nabla_\omega T_{M[u]^{-1}}v
+\begin{pmatrix}
\xi \\ \mu
\end{pmatrix}
=f
\end{equation}
has a linear solution operator 
$$
\big(
v^x, \,
v^y, \,
\xi, \,
\mu
\big)
=\big(
\mathfrak{L}^x[u]f, \,
\mathfrak{L}^y[u]f, \,
P^x[u]f, \,
P^y[u]f
\big),
$$
satisfying the following estimates: with $K_s$ being increasing functions in the arguments,
$$
\begin{aligned}
\big\|\mathfrak{L}^x[u]f\big\|_{H^s}
&\leq K_s\big(M_Q,|h|_{C^3}\big)
\gamma^2\|f\|_{H^{s+2\sigma}} \\
\big\|\mathfrak{L}^y[u]f\big\|_{H^s}
&\leq K_s\big(M_Q,|h|_{C^3}\big)\gamma\|f\|_{H^{s+\sigma}} \\
|P^x[u]f|+|P^y[u]f|
&\leq K_s\big(M_Q,|h|_{C^3}\big)\gamma\|f\|_{H^{s+\sigma}}.
\end{aligned}
$$
Moreover, the four linear operators of concern are all continuously differentiable mappings from $u\in C^1$ to the space of linear operators (with operator norm). In particular, the solution operator $\xi=P^x[u]f$ is fixed as 0 under the assumption of Theorem \ref{Thm1}.
\end{lemma}
\begin{proof}
Since $M[u]=\mathrm{diag}(I_n,-I_n)+O\big(\partial(u-\zeta_0)\big)$, it follows that $T_{M[u]}$ and $T_{M[u]^{-1}}$ are both invertible bounded operator from $H^s$ to $H^s$ if $|u-\zeta_0|_{C^1}$ is small. Note further that for a constant $\lambda$ and a function $A$, there holds $T_A\lambda=(\Avg A)\lambda$. Thus, writing
\begin{equation}\label{munu}
T_{M[u]}^{-1}\begin{pmatrix}
\xi \\ \mu
\end{pmatrix}=\begin{pmatrix}
\xi_1 \\ \mu_1
\end{pmatrix},
\quad
T_{M[u]}^{-1}f=f_1,
\quad 
T_{M[u]^{-1}}v=v_1
\end{equation}
for some constant vectors $\xi_1,\mu_1\in\xR^n$ to be determined, the linear para-homological equation (\ref{LinParaHomo'}) is equivalent to
$$
\begin{pmatrix}
0_n & T_{S[u]} \\
0_n & 0_n
\end{pmatrix}v_1
-\nabla_\omega v_1
+\begin{pmatrix}
\xi_1 \\ \mu_1
\end{pmatrix}
=f_1.
$$
Furthermore, recalling the expression (\ref{h(x,y)}) of $h(x,y)$ and (\ref{S[u]}), we have
\begin{equation}\label{S[u]'}
S[u]=Q+\partial^2a_0+O\big(\partial(u-\zeta_0)\big),
\end{equation}
where the implicit constant depends on $|h|_{C^3}$. Noting that
\begin{equation}\label{e_0}
e_0=\mathscr{F}(h,\zeta_0)
=\begin{pmatrix}
a_1-\omega \\
\partial a_0
\end{pmatrix},
\end{equation}
it follows that $S[u]$ is $C^0$-close to $Q$ if $|e_0|_{C^1}$ and $|u-\zeta_0|_{C^1}$ are small. 

In components of $v_1$, the linear para-homological equation becomes
$$
\begin{aligned}
T_{S[u]}v_1^y-\nabla_\omega v_1^x+\xi_1&=f_1^x, \\
-\nabla_\omega v_1^y+\mu_1&=f_1^y.
\end{aligned}
$$
We discuss how to solve this system under the assumption of Theorem \ref{Thm1} and \ref{Thm2} separately. For simplicity of notation, we define
$$
\mathcal{A}f:=f-\Avg f.
$$

\subsubsection*{Case 1: $\Avg Q$ is Invertible}
This is the non-degeneracy condition in Therorem \ref{Thm1}. We fix $\xi=0$ in this case, and obtain from (\ref{munu}) the following:
\begin{equation}\label{munu1}
\begin{pmatrix}
0 \\ \mu
\end{pmatrix}
=\begin{pmatrix}
\xi_1 \\ -\mu_1
\end{pmatrix}
+O\big(\partial(u-\zeta_0)\big)\begin{pmatrix}
\xi_1 \\ \mu_1
\end{pmatrix}.
\end{equation}
We first solve 
\begin{equation}\label{Vymu1}
v_1^y=\Avg v_1^y-\nabla_\omega ^{-1}\mathcal{A}f_1^y,
\quad
\mu_1=\Avg f_1^y,
\end{equation}
where $\Avg v_1^y$ is yet to be determined. With (\ref{munu1}), one can solve $\xi_1$, thus $\mu$, as linear mapping of $\mu_1$. Since the norm of $T_{M[u]}^{-1}$ from $H^s$ to $H^s$ depends only on $|u-\zeta_0|_{C^1}$, this immediately yields
$$
\big\|\mathcal{A}v_1^y\big\|_{H^s}+|\mu|+|\xi_1|
\leq C_s\gamma\|f\|_{H^{s+\sigma}}.
$$
Here we used (\ref{ContHs}).

We then choose $\Avg v_1^y$ to balance the mean value of the equation along $x$-direction. By assumption $S[\zeta_0]$ is invertible, so by (\ref{S[u]'}), the matrix-valued function $S[u]=Q+O(\partial e_0)+O\big(\partial(u-\zeta_0)\big)$ is still invertible if $|e_0|_{C^1}$ and $|u-\zeta_0|_{C^1}$ are small, and the implicit constant here depends on $|h|_{C^3}$. In that case, $\big|(\Avg S[u])^{-1}\big|\lesssim M_Q$. We can then set
$$
\Avg v_1^y=\big(\Avg S[u]\big)^{-1}\Avg\big(f_1^x-\xi_1-T_{S[u]}\mathcal{A}v_1^y\big),
$$
and solve
$$
v_1^x=\nabla_\omega ^{-1}\mathcal{A}
\big(T_{S[u]}\mathcal{A}v_1^y+\xi_1-f_1^x\big).
$$
By (\ref{S[u]}), $S[u]$ depends on up to second order derivative of $h$, so the solutions are immediately estimated as
$$
\big\|v_1^x\big\|_{H^s}
\leq K_s\big(M_Q,|h|_{C^3}\big)
\gamma^2\|f\|_{H^{s+2\sigma}}.
$$
Here we used (\ref{ContHs}) once again.

\subsubsection*{Case 2: Non-trivial Parameter $\xi$}
This case matches with Theorem \ref{Thm2}, whence no requirement about non-degeneracy of $Q$ is posed. In this case the parameter $\xi$ cannot be fixed as 0. But we still obtain from (\ref{munu}) the following:
\begin{equation}\label{munu2}
\begin{pmatrix}
\xi \\ \mu
\end{pmatrix}
=\begin{pmatrix}
\xi_1 \\ -\mu_1
\end{pmatrix}
+O\big(\partial(u-\zeta_0)\big)\begin{pmatrix}
\xi_1 \\ \mu_1
\end{pmatrix}.
\end{equation}
We then first solve 
\begin{equation}\label{Vymu2}
v_1^y=-\nabla_\omega ^{-1}\mathcal{A} f_1^y,
\quad
\mu_1=\Avg f_1^y,
\end{equation}
i.e. the mean value of $v_1^y$ is fixed as 0. For the equation along $x$-direction, we solve 
$$
v_1^x=-\nabla_\omega ^{-1}\mathcal{A} \big(f_1^x-T_{S[u]}v_1^y\big),
\quad
\xi_1=\Avg \big(f_1^x-T_{S[u]}v_1^y\big).
$$
Using (\ref{ContHs}), they immediately yield the similar estimates as in Case 1.

Finally, continuous differentiability of the operators for $u\in C^1$ follows immediately from Proposition \ref{ContPM} and \ref{ContCM}, that is, continuous differentiability of $T_a$ in $a$ and $T_{ab}-T_aT_b$ in $(a,b)$.
\end{proof}

\subsection{Proof of Theorem \ref{Thm1}-\ref{Thm2}}\label{ProofofThm}
Recall that we set $r=s-n/2$, and $N_{s+r}$ to be the least integer $>s+r$. Since we are working with a perturbative problem, we assume that e.g. $\|u-\zeta_0\|_{H^s}\leq2$, and the requirement of Lemma \ref{LinParaHomo} is also fulfilled: with $\rho_1=\rho_1\big(M_Q,|h|_{C^3}\big)$, $\rho_2=\rho_2\big(M_Q,|h|_{C^3}\big)$ as in Lemma \ref{LinParaHomo}, we have
\begin{equation}\label{C1bound}
\|u-\zeta_0\|_{H^s}\leq2,
\quad
|e_0|_{C^1}\leq \rho_1,\quad |u-\zeta_0|_{C^1}\leq\rho_2.
\end{equation}

The para-homological equation (\ref{ParaHomoHamilton}) is then converted to the following para-inverse equation:
\begin{equation}\label{ParaInvHamilton}
\begin{aligned}
u-\zeta_0
&=-\begin{pmatrix}
\mathfrak{L}^x[u]-P^x[u] \\
\mathfrak{L}^y[u]-P^y[u] 
\end{pmatrix}e_0\\
&\quad
-\begin{pmatrix}
\mathfrak{L}^x[u]-P^x[u] \\
\mathfrak{L}^y[u]-P^y[u]
\end{pmatrix}\left(\R_{\CM}[u](u-\zeta_0)
-\R_{\PL}\big(X_h(\zeta_0+\cdot\,),u-\zeta_0\big)(u-\zeta_0)\right)\\
&=:\mathcal{G}_1(u)+\mathcal{G}_2(u)
\end{aligned}
\end{equation}

We analyze the right-hand of (\ref{ParaInvHamilton}). The bound (\ref{C1bound}) meets the requirement of Lemma \ref{LinParaHomo}, so
\begin{equation}\label{G_1}
\begin{aligned}
\|\mathcal{G}_1(u)\|_{H^{s+\eps}}
&\leq K_s\big(M_Q,|h|_{C^3}\big)\gamma^2\|e_0\|_{H^{s+2\sigma+\eps}}.
\end{aligned}
\end{equation}

We turn to $\mathcal{G}_2$. The para-product smoothing operator $\R_{\CM}[u]$ in $\mathcal{G}_2$ has expression (\ref{RCM}). We recall that $M[u]=\mathrm{diag}(I_n,-I_n)+O\big(\partial(u-\zeta_0)\big)$, while obviously 
$$
\big|S[u]\big|_{C^{r-1}_*}\quad\text{and}\quad
\big|\partial M[u]^{-1}\big|_{C^{r-2}_*}
$$
are both controlled by increasing functions of $|h|_{C^{N_{s+r}}}$ and $\|u\|_{H^s}$. Using Proposition \ref{ContCM}, noting $M[u]$ is almost a constant matrix, it follows that
\begin{equation}\label{RCMEst}
\big\|\R_{\CM}[u](u-\zeta_0)\big\|_{H^{s+r-2}}
\leq K_s\big(|h|_{C^{N_{s+r}}}\big)\|u-\zeta_0\|_{H^s}^2,
\end{equation} 
and $\R_{\CM}[u](u-\zeta_0)$ depends continuously on $u$ (with respect to Sobolev norms $H^s$ and $H^{s+r-2}$).

For the para-linearization remainder $\R_{\PL}$, by Proposition \ref{Bony}, recalling (\ref{A[u]}),
\begin{equation}\label{RPLEq}
\R_{\PL}\big(X_h(\zeta_0+\cdot\,),u-\zeta_0\big)(u-\zeta_0)
=\R_{\PL;1}\big(A[\zeta_0]\big)(u-\zeta_0)
+\R_{\PL;2}\big(X_h(\zeta_0+\cdot\,),u-\zeta_0\big)(u-\zeta_0).
\end{equation}
Recalling further the definition (\ref{h(x,y)}) of $h(x,y)$, we find 
$$
A[\zeta_0]
=\begin{pmatrix}
\partial a_1(\theta) & Q(\theta) \\
\partial^2a_0(\theta) & -(\partial a_1)^\T(\theta) 
\end{pmatrix}.
$$
On the other hand, recalling the definition (\ref{e_0}), we find
$$
\begin{aligned}
\|a_1-\omega\|_{H^{s+\sigma+\eps}}
+\|\partial a_0\|_{H^{s+\sigma+\eps}}
&=\|e_0\|_{H^{s+2\sigma+\eps}},
\end{aligned}
$$
that is, $a_0$ is almost a constant, and $a_1$ almost equals the constant vector $\omega$. As a result, we find $A[\zeta_0]$ is almost a constant:
$$
\begin{aligned}
\big|A[\zeta_0]-\Avg A[\zeta_0]\big|_{C^r_*}
&\leq C_s\big(\|a_1-\omega\|_{H^{s+\sigma+\eps}}
+\|\partial a_0\|_{H^{s+\sigma+\eps}}
+|Q-\Avg Q|_{C^r_*}\big)\\
&\leq C_s\big(\|e_0\|_{H^{s+2\sigma+\eps}}+|e_1|_{C^r_*}\big).
\end{aligned}
$$
Implementing the para-linearization theorem i.e. Proposition \ref{Bony}, we estimate (\ref{RPLEq}) as (note that $X_h=J\nabla_{x,y}h\in C^{N_{s+r}+2}$)
\begin{equation}\label{RPLEst}
\begin{aligned}
\big\|\R_{\PL}&\big(X_h(\zeta_0+\cdot\,),u-\zeta_0\big)(u-\zeta_0)\big\|_{H^{s+r}} \\
&\leq C_s\Big(\big(\|e_0\|_{H^{s+2\sigma+\eps}}+|e_1|_{C^r_*}\big)\|u-\zeta_0\|_{H^s}
+|h|_{C^{N_{s+r}+3}}\|u-\zeta_0\|_{H^s}^2\Big).
\end{aligned}
\end{equation}
Furthermore, by Proposition \ref{Bony}, the operator $\R_{\PL}\big(X_h(\zeta_0+\cdot\,),u-\zeta_0\big)\in\mathcal{L}(H^s,H^{s+r})$ depends continuously on $u\in H^s$.

Recall that by assumption, $s+r-2=2s-n/2-2\geq s+2\sigma+\eps$. Combining (\ref{G_1})(\ref{RCMEst})(\ref{RPLEst}), using Lemma \ref{LinParaHomo}, we conclude that if (\ref{C1bound}) holds, then
\begin{equation}\label{G1G2}
\begin{aligned}
\|\mathcal{G}_1(u)&\|_{H^{s+\eps}}
+\|\mathcal{G}_2(u)\|_{H^{s+\eps}}\\
&\leq K_s\big(M_Q,|h|_{C^{N_{s+r}+3}}\big)\gamma^{2}
\Big(\|e_0\|_{H^{s+2\sigma+\eps}}+\big(\|e_0\|_{H^{s+2\sigma+\eps}}+|e_1|_{C^r_*}\big)\|u-\zeta_0\|_{H^s}
+\|u-\zeta_0\|_{H^s}^2\Big).
\end{aligned}
\end{equation}
By some elementary algebra with quadratic polynomials, there are decreasing functions $c_0,c_1$ and an increasing function $c_2$, with the following property: if 
$$
\|e_0\|_{H^{s+2\sigma+\eps}}
\leq\gamma^{-4}c_0\big(M_Q,|h|_{C^{N_{s+r}+3}}\big),
\quad
|e_1|_{C^r_*}
\leq \gamma^{-2}c_1\big(M_Q,|h|_{C^{N_{s+r}+3}}\big),
$$
then the positive number $\rho=\|e_0\|_{H^{s+\sigma+\eps}}\gamma^{-2}c_2\big(M_Q,|h|_{C^{N_{s+r}+3}}\big)$ satisfies the quadratic inequality
$$
K_s\big(|h|_{C^{N_{s+r}+3}}\big)\gamma^2
\Big(\|e_0\|_{H^{s+2\sigma+\eps}}+\big(\|e_0\|_{H^{s+2\sigma+\eps}}+|e_1|_{C^r_*}\big)\rho
+\rho^2\Big)\leq\rho.
$$
Here $K_s$ is the increasing function in (\ref{G1G2}). In other words, 
$$
\|\mathcal{G}_1(u)\|_{H^{s+\eps}}
+\|\mathcal{G}_2(u)\|_{H^{s+\eps}}
\leq \rho
\quad\text{if }\|u-\zeta_0\|_{H^s}\leq\rho\text{ and (\ref{C1bound}) holds.}
$$
In that case, the mapping $\zeta_0+\mathcal{G}_1+\mathcal{G}_2$ in the right-hand-side of (\ref{ParaInvHamilton}) maps the closed ball $\bar B_{\rho}(\zeta_0)\subset H^s$ to itself. Furthermore it is continuous with respect to $u\in \bar B_{\rho}(\zeta_0)\subset H^s$ and the image is totally bounded (since the embedding $H^{s+\eps}\subset H^s$ is compact). By the Schauder fixed point theorem, the mapping $u\mapsto\zeta_0+\mathcal{G}_1(u)+\mathcal{G}_2(u)$  has a fixed point $u\in\bar B_{\rho}(\zeta_0)$, and in fact
\begin{equation}\label{ubound}
\|u-\zeta_0\|_{H^s}
\leq c_2\big(M_Q,|h|_{C^{N_{s+r}+3}}\big)\gamma^2\|e_0\|_{H^{s+2\sigma+\eps}}.
\end{equation}

Having found a solution $u$ of the para-homological equation (\ref{ParaHomoHamilton}), we determine the constants $\xi,\eta$ in terms of $u$ via the operators $P^x[u],P^y[u]$. We will then be able to cancel all terms but the one linear in $E$ in the para-linearization formula (\ref{PLF(h,u)}). Thus the para-linearization formula (\ref{PLF(h,u)}) at $u$ reduces to
\begin{equation}\label{E=T_Eu}
E=\Op^\PM\left(\bm{B}(E,L[u]) M[u]^{-1}\right)(u-\zeta_0).
\end{equation}
Using Corollary \ref{B(E,L)}, we estimate
\begin{equation}\label{Est_B(E,L)}
\begin{aligned}
\big\|\Op^\PM&\left(\bm{B}(E,L[u]) M[u]^{-1}\right)(u-\zeta_0)\big\|_{H^s}\\
&\overset{\text{Prop. }\ref{ContPM}}{\leq} C_s\big|\bm{B}(E,L[u]) M[u]^{-1}\big|_{L^\infty}\|u-\zeta_0\|_{H^s}\\
&\overset{(\ref{ubound})}{\leq} C_sc_2\big(M_Q,|h|_{C^{N_{s+r}+3}}\big)\gamma^2\|e_0\|_{H^{s+2\sigma+\eps}}\big\|\bm{B}(E,L[u]) M[u]^{-1}\big\|_{H^{s-(\sigma+2)}}\\
&\overset{\text{Cor. }\ref{B(E,L)}}{\leq}
C_sc_2\big(M_Q,|h|_{C^{N_{s+r}+3}}\big)\gamma^3\|e_0\|_{H^{s+2\sigma+\eps}}\|E\|_{H^{s-1}}.
\end{aligned}
\end{equation}
Note that it is legitimate to apply Corollary \ref{B(E,L)}, since we already know that $u$ is close to $\zeta_0$, and therefore $E=\mathscr{F}(h_\xi,u)+(0,\mu)^\T$ is close to zero. Consequently, if we further require the function $c_0\ll 1/c_2$, then the factor 
$$
C_sc_2\big(M_Q,|h|_{C^{N_{s+r}+3}}\big)\gamma^3\|e_0\|_{H^{s+2\sigma+\eps}}
$$
in the right-hand-side of (\ref{Est_B(E,L)}) will be $\leq 1/2$ (remembering $e_0$ is of size $O(\gamma^{-4}c_0)$ and $\gamma$ is assumed to be large). Therefore, (\ref{Est_B(E,L)}) yields
$\|E\|_{H^s}\leq (1/2)\|E\|_{H^{s-1}}$, forcing $E=0$. 

To summarize, under the assumption of Theorem \ref{Thm1}, there is a solution $u$ of the the para-homological equation (\ref{ParaHomoHamilton}) with $\xi=0$ in (\ref{ParaHomoHamilton}), the constant vector $\mu$ is determined by $u$, and $u$ satisfies 
$$
E=\mathscr{F}(h,u)+\begin{pmatrix} 0 \\ \mu\end{pmatrix}=0.
$$
Under the assumption of Theorem \ref{Thm2}, there is a solution $u$ of the the para-homological equation (\ref{ParaHomoHamilton}), the parameters $\xi,\mu$ are determined by $u$, and $u$ satisfies 
$$
E=\mathscr{F}(h_\xi,u)+\begin{pmatrix} 0 \\ \mu\end{pmatrix}=0.
$$
By Lemma \ref{CounterTerm}, the proof of Theorerm \ref{Thm1}-\ref{Thm2} is complete.

\begin{remark}
Similarly as in Subsection \ref{ParaDiffIllu}, the right-hand-side of the para-inverse equation (\ref{ParaInvHamilton}) becomes continuously differentiable in $u$ if we assume higher regularity $h\in C^{N_{s+r}+4}$. This is indicated in the general para-linearization theorem \ref{Bony}. Thus if the errors $e_0$ and $e_1$ are even smaller, the right-hand-side of (\ref{ParaInvHamilton}) shall have Lipschitz constant less than 1, and a Banach fixed point argument will apply. In this case, the KAM conjugacy problem falls into the realm of standard implicit function theorems, which will immediately imply, for example, Whiteney differentiablity with respect to the frequency $\omega$. Compared to what is in the literature, this direct proof is much simpler.
\end{remark}

\subsection{Miscellaneous}\label{Miscellaneous}
We discuss the generality of the assumption of Theorem \ref{Thm1}-\ref{Thm2}. The idea of our discussion closely follows \cite{BB2015}. At a first glance, the coverage of Theorem \ref{Thm1}-\ref{Thm2} is very limited: the approximate invariant torus around which we seek for a true one has to be the ``flat" embedding $\theta\mapsto(\theta,0)$. However, under suitable geometric transformations, which are elementary and independent from KAM type results, it is possible to reduce quite general situations to this special case.

To be concrete, let the frequency $\omega$ satisfy (\ref{Dio}). Suppose $h(x,y)$ is an arbitrary Hamiltnonian function, not necessarily brought into the form (\ref{h(x,y)}). Let
$$
\zeta(\theta)=\begin{pmatrix}
\zeta^x(\theta) \\
\zeta^y(\theta)
\end{pmatrix}
$$
be an embedding from $\xT^n$ to the phase space, not necessarily close to the ``flat" embedding, such that $\zeta^x:\xT^n\mapsto\xT^n$ is a diffeomorphism, and $\zeta(\xT^n)$ is an approximately invariant torus with frequency $\omega$ under the flow of $h$: that is,
$$
\mathscr{F}(h,\zeta):=X_h(\zeta)-\nabla_\omega \zeta
$$
is close to 0. We aim to find suitable symplectic coordinate transformation that reduce this general approximate solution to the special case discussed in Theorem \ref{Thm1}-\ref{Thm2}.

Recall (\ref{L[u]}): we introduce
$$
L[\zeta]=(\partial \zeta)^\T \cdot J \partial \zeta\in\bm{M}_{n\times n}
$$
to measure the ``lack of isotropy" of the embedding $\zeta$, since the pull-back of the symplectic form $\diff\!x\wedge \diff\!y$ to $\xT^n$ via $\zeta$ has exactly the matrix representation $L[\zeta]$. We then have the following, as seen in (\ref{Lack}):
\begin{equation}\label{L[zeta]}
L[\zeta]=\partial\nabla_\omega ^{-1}\left[(\partial \zeta)^\T\cdot J\mathscr{F}(h,\zeta)\right]
-\partial\nabla_\omega ^{-1}\left[(\partial \zeta)^\T\cdot J\mathscr{F}(h,\zeta)\right]^\T.
\end{equation}

According to Lemma 6 of \cite{BB2015}, if one sets
$$
p(\theta)=\Delta^{-1}\partial\cdot L[\zeta]
=\left(\Delta^{-1}\sum_{j=1}^n\partial_jL_{kj}(\theta)\right)_{1\leq k\leq n},
$$
then the embedded torus
$$
\eta(\theta)
=\begin{pmatrix}
\eta^x(\theta) \\
\eta^y(\theta)
\end{pmatrix}
:=\begin{pmatrix}
\zeta^x(\theta) \\
\zeta^y(\theta)-\big(\partial\zeta^x(\theta)\big)^\T p(\theta)
\end{pmatrix}
$$
is isotropic. From (\ref{L[zeta]}), $p$, hence $\eta-\zeta$, is linear in $\mathscr{F}(h,\zeta)$, and for $s>1+n/2$,
$$
\|\eta-\zeta\|_{H^s}\leq C_s\gamma\|\zeta\|_{H^{s+\sigma+1}}\|\mathscr{F}(h,\zeta)\|_{H^{s+\sigma}},
$$
where $C_s$ only depends on $s$. Since
$$
\begin{aligned}
\mathscr{F}(h,\eta)
&=X_h\big(\zeta+(\eta-\zeta)\big)-\nabla_\omega \zeta
-\nabla_\omega (\eta-\zeta)\\
&=\mathscr{F}(h,\zeta)
+\int_0^1 (DX_h)\big(\zeta+\tau(\eta-\zeta)\big)(\eta-\zeta)\diff\!\tau
-\nabla_\omega (\eta-\zeta),
\end{aligned}
$$
it then follows that
$$
\|\mathscr{F}(h,\eta)\|_{H^s}
\leq C_s\gamma\big(1+|h|_{C^{N_{s+r}+3}}\big)
\big(1+\|\zeta\|_{H^{s+\sigma+2}}\big)\|\mathscr{F}(h,\zeta)\|_{H^{s+\sigma+1}},
$$
where $C_s$ only depends $s$ and $|\omega|$. In other words, if $\zeta(\xT^n)$ is an approximately invariant torus for $X_h$, then so is the isotropic torus $\eta(\xT^n)$, but the ``error of being invariant" is less regular compared to $\zeta(\xT^n)$.

Since $\eta$ is an isotropic embedding, the diffeomorphism in phase space
$$
G:\begin{pmatrix}
x \\
y
\end{pmatrix}\mapsto
\begin{pmatrix}
\eta^x(x) \\
\eta^y(x)+\big(\partial\eta^x(x)\big)^\T y
\end{pmatrix}
$$
is a symplectic one. The isotropic torus $\eta(\xT^n)$ is straightened to the ``flat" torus $\{y_1=0\}$ under the new symplectic coordinate $(x_1,y_1)=G^\iota(x,y)$. If we set
$$
H(x_1,y_1)
:=(h\circ G)(x_1,y_1)
=a_0(x_1)+\langle a_1(x_1),y_1\rangle
+\frac{1}{2}\langle Q(x_1)y_1,y_1\rangle+O(|y_1|^3)
\quad\text{near }y=0,
$$
then the Hamiltonian vector field $X_h$ is changed to $X_H=(DG)^{-1}X_h\circ G$. Furthermore, the ``error of being invariant" is transformed to
$$
(DG)^{-1}\mathscr{F}(h,\eta)
=X_H(\zeta_0)-\begin{pmatrix}\omega \\ 0 \end{pmatrix}
=\begin{pmatrix}
a_1(\theta)-\omega \\ 
\partial a_0(\theta) \end{pmatrix}.
$$
Thus $a_0$ is almost constant, and $a_1$ is almost equal to $\omega$. We then set $S=\nabla_\omega ^{-1}(Q-\Avg Q)$, and introduce another canonical variable $(x_2,y_2)$, defined by
$$
\begin{aligned}
x_1&=x_2+S(x_1)y_2\\
y_1&=y_2-\frac{1}{2}\partial S(x_1)(y_2,y_2).
\end{aligned}
$$
The symplectic diffeomorphism $\Gamma$ maps the isotropic torus $\{y_2=0\}$ to $\{y_1=0\}$, while the Hamiltonian function 
$$
\begin{aligned}
H(x_1,y_1)
&=\left(a_0(x_1)-\Avg a_0\right)
+\langle a_1(x_1)-\omega,y_1\rangle\\
&\quad+\Avg a_0+\langle \omega,y_2\rangle
+\frac{1}{2}\langle (\Avg Q)y_2,y_2\rangle+O(|y_2|^3).
\end{aligned}
$$
We thus find that the first three Taylor coefficients of $H(x_1,y_1)$ with respect to $y_2$ around $y_2=0$ will be close to 
$$
\Avg a_0,\quad \omega,\quad \Avg Q,
$$
respectively. Consequently, under the new symplectic coordinate $(x_2,y_2)$, the Hamiltonian function and approximately invariant torus are exactly in the form discussed by Theorem \ref{Thm1}-\ref{Thm2}.

In conclusion, we have successfully reduced the general situation to the special case of Theorem \ref{Thm1} without any assumption on smallness of $\zeta^x,\zeta^y$. For Hamiltonian functions depending on parameters, the argument will be exactly the same. 

\appendix
\section{``Direct Method" for Conjugacy Problems}\label{AppA}
In this appendix, we explain why we prefer the ``indirect method" to study conjugacy problems by comparing it with the ``direct method", to be specified shortly. Since this is is only a complementary exposition, we will not be concerned with the choice of appropriate function spaces. The somewhat formal computation already suggests that the ``direct method" is technically much more cumbersome compared to the ``indirect method".

We first sketch how the conjugacy problem for circular maps can be solved in a ``direct" way. The idea seems to come from the lecture notes of H\"{o}rmander \cite{Hormander1977}. Define a nonlinear mapping
$$
\mathscr{G}(u,\lambda)=\big[\Delta_\alpha u\big]\circ(\Id+u)^\iota+\lambda,
$$
which appears in (\ref{ConjRot1}). Computing the linearization of $\mathscr{G}$ at $(u,\lambda)$ close to $(0,0)$, we find
\begin{equation}
\begin{aligned}
D\mathscr{G}(u,\lambda)(v,\mu)
&=\big[\Delta_\alpha v\big]\circ(\Id+u)^\iota\\
&\quad+\left(\big[\Delta_\alpha(u+tv)'\big]\circ(\Id+u+tv)^\iota\right)\cdot\frac{d}{dt}(\Id+u+tv)^\iota\Big|_{t=0}
+\mu \\
&=\left[(1+u'\circ\tau_\alpha)\Delta_\alpha\left(\frac{v}{1+u'}\right)\right]\circ(\Id+u)^\iota+\mu.
\end{aligned}
\end{equation}
The linearized equation $D\mathscr{G}(u,\lambda)(v,\mu)=h$ is then solved by
\begin{equation}\label{RotInvExact}
v=(1+u')\Delta_\alpha^{-1}\left(\frac{h\circ(\Id+u)-\mu}{1+u'\circ \tau_\alpha}\right),
\end{equation}
where $\mu$ is the unique real number making the function inside the bracket to have average 0. Thus, if we consider equation (\ref{ConjRot1}) instead of (\ref{ConjRot0}), then \emph{the linearized equation is exactly solvable}.

Note that product and composition of mappings in the grading $\cap_{r\geq0}C_*^r$ satisfy tame estimates. Using these tame estimates, we find that the solution $(v,\mu)$ of the linearized equation $D\mathscr{G}(u,\lambda)(v,\mu)=h$ satisfies the following tame estimate:
$$
|v|_{C_*^r}+|\mu|
\lesssim_r \big(1+|u|_{C_*^{2+1/2}}\big)|h|_{C_*^{r+s}}+|h|_{C_*^{1/2}}\big(1+|u|_{C_*^{r+s+1}}\big),
\quad
\text{where }s>\sigma+1.
$$
We can finally apply any Nash-Moser type theorem to conclude that, at least for very regular $f$ with small magnitude, there is a solution $u$ of (\ref{ConjRot1}). 

In this case, the para-differential approach is still valid to produce an alternative proof. We suppose in apriori that $u$ is sufficiently smooth. For simplicity of notation, write $W=W(u)=(\Id+u)^\iota$. Let us try to decompose the nonlinear mapping
$$
u\mapsto
(\Delta_\alpha u)\circ W
=(\Delta_\alpha u)\circ(\Id+u)^\iota
$$
in the left-hand-side of (\ref{ConjRot1}) using appropriate para-differential operators. By Probosition \ref{PLCRough},
\begin{equation}\label{RotCompo0}
(\Delta_\alpha u)\circ W
=W^\star\Delta_\alpha u
+T_{(\Delta_\alpha u')\circ W}W
+\text{smoother remainder}.
\end{equation}
We leave the first term in (\ref{RotCompo0}) unchanged and concentrate on the para-product term, aiming to express $W(u)$ in terms of the para-composition operator. In order to do this, we para-linearize the identity $(\Id+u)\circ W(u)=\Id$ to find
$$
W(u)+W^\star u+T_{u'\circ W}W(u)
=\Id+\text{smoother remainder}.
$$
Solving for $W(u)$, we thus obtain
\begin{equation}\label{RotInv}
\begin{aligned}
W(u)=-T_{(1+u'\circ W)^{-1}}W^\star u+\Id+\text{smoother remainder}.
\end{aligned}
\end{equation}
We then cite a proposition on conjugation of para-differential operators using para-composition, which is a special case of the conjugation theorem proved in \cite{Alinhac1986,ngu2016} or \cite{Said2023}:
\begin{proposition}
Let $r>0,\rho>1$. Suppose $a(x,\xi)=\sum_{|\alpha|\leq m}a_\alpha(x)(i\xi)^\alpha$ is a classical symbol with $C^r_*$ coefficients on $\xT^n$. Suppose $\chi:\xT^n\mapsto\xT^n$ is a $C^\rho_*$ diffeomorphism. Then
$$
\chi^\star T_a=T_{a^\star}\chi^\star+\R_{\mathrm{conj}}(\chi,a),
$$
where 
$$
a^*(x,\xi)
:=\sum_{|\beta|=0}^{[\rho]}\frac{1}{i^{|\beta|}|\beta|!}\partial_y^\beta\partial_\xi^\beta
\left(a\big(\chi(x),R(x,y)^{-1}\xi\big)
\left|\frac{\det\chi'(y)}{\det R(x,y)}\right|\right)
\Bigg|_{y=x},
$$
with 
$$
R(x,y)=\int_0^1\chi'(tx+(1-t)y)^{\mathrm{T}}dt,
$$
while the remainder operator maps $H^s$ to $H^{s-m+\min(r,\rho-1)})$, and satisfies
$$
\|\R_{\mathrm{conj}}(\chi,a)\|_{\mathcal{L}(H^s,H^{s-m+\min(r,\rho-1)})}
\leq K_s\big(|\chi|_{C^\rho_*}\big)\sum_{|\alpha|\leq m}|a_\alpha|_{C^r_*}.
$$
\end{proposition}
Thus, (\ref{RotCompo0}) is transformed to
\begin{equation}\label{RotCompo1}
\begin{aligned}
(\Delta_\alpha u)\circ W
&=W^\star\Delta_\alpha u
-T_{(\Delta_\alpha u')\circ W}T_{(1+u'\circ W)^{-1}}W^\star u
+\text{smoother remainder}\\
&=W^\star\Delta_\alpha u
-W^\star T_{(\Delta_\alpha u')/(1+u')}u
+\text{smoother remainder}\\
&=W^\star T_{(1+u'\circ\tau_\alpha)}\Delta_\alpha T_{1/(1+u')}u
+\text{smoother remainder}.
\end{aligned}
\end{equation}
The original nonlinear equation (\ref{ConjRot1}) then becomes
\begin{equation}\label{RotEquiv}
W^\star T_{(1+u'\circ\tau_{\alpha})}\Delta_\alpha T_{1/(1+u')}u
+\lambda
=f+\text{smoother remainder},
\end{equation}
which may still be named \emph{para-homological equation}. (\ref{RotEquiv}) can still be converted to a fixed point type equation:
\begin{equation}\label{RotEquivDirect}
u=T_{1/(1+u')}^{-1}\Delta_\alpha^{-1}T_{(1+u'\circ\tau_{\alpha})}^{-1}(W^\star)^{-1}\big(f-\lambda-\text{smoother remainder}\big).
\end{equation}
The $\lambda$ can be uniquely determined by balancing the mean value, so that $\Delta_\alpha^{-1}$ can be applied. 

(\ref{RotCompo1}) is exactly the \emph{para-differential version} of the linearization of $\mathscr{G}(u,\lambda)$, or to put simply, the \emph{para-linearization} of $\mathscr{G}(u,\lambda)$. Although quite parallel to the ``indirect method", we still prefer the ``indirect method" instead. The reason is that the fixed point equation produced by the ``indirect method" is technically less involved than the ``direct method". Even in this simple illustrative model, the fixed point equation (\ref{RotEquivDirect}) requires more para-composition estimates compared to (\ref{RotParaInv}): it calls for knowledge about inversion of para-compositions, together with conjugation with para-compositions, while (\ref{RotParaInv}) relies only on the para-linearization formula.

Let us finally discuss the Hamiltonian conjugacy problem. As pointed out by Féjoz \cite{Fejoz2016}, given a Hamiltonian function $h$ close to a normal form, if one aims to find a normal form $H$, a symplectic diffeomorphism $\Gamma$ and a frequency shift $\beta$ solving
\begin{equation}\label{HermanNormalApp0}
h=H\circ\Gamma+\beta\cdot y,
\end{equation}
then the linearized equation will be exactly solvable in a neighbourhood of $(\Gamma,\beta,H)=(\Id,0,h_0)$. On the other hand, if one considers instead
$$
h\circ\Gamma-\beta\cdot (y\circ\Gamma)=H,
$$
the linearized equation will only admit approximate solution. The form (\ref{HermanNormalApp0}) then brings a lot of convenience for the Nash-Moser scheme. However, as with the circular map problem discussed earlier, the para-linearization of (\ref{HermanNormalApp0}) is technically more complicated compared to the ``indirect method" since it inevitably involves numbers of para-composition estimates, though the general idea remains unchanged. In summary, we would prefer the ``indirect method" discussed in this paper instead, where we solve the para-inverse equation and conclude that it annihilates the nonlinear mapping by a Neumann series argument.

\bibliographystyle{alpha}
\bibliography{References}

\end{spacing}
\end{document}